\documentclass[a4paper,12pt,reqno]{amsart}
\usepackage{a4wide}
\usepackage{xcolor}
\usepackage{graphicx}
\usepackage{amsmath}
\usepackage{amssymb}
\usepackage{esint}
\usepackage{hyperref}

\theoremstyle{plain}
\newtheorem{theorem}{Theorem}[section]
\newtheorem{proposition}[theorem]{Proposition}
\newtheorem{lemma}[theorem]{Lemma}
\newtheorem{corollary}[theorem]{Corollary}
\newtheorem{remark}[theorem]{Remark}

\newtheorem{definition}[theorem]{Definition}

\newcommand{\C}{\mathbb{C}}

\newcommand{\R}{\mathbb{R}}

\newcommand{\const}{c}

\newcommand{\eps}{\varepsilon}

\let\altphi\phi
\let\phi\varphi
\let\varphi\altphi
\let\altphi\undefined

\renewcommand{\div}{{\rm div}\,}

\begin{document}
	
	\title[Migdal's momentum loop equation]{Some rigorous remarks on Migdal's momentum loop equation}

\author[E. Bru\'e]{Elia Bru\'e}
\address{Department of Decision Sciences, Bocconi University, Via Sarfatti 25 20136 Milano (MI), Italy}
\email{elia.brue@unibocconi.it}
\author[C. De Lellis]{Camillo De Lellis}
\address{School of Mathematics, Institute for Advanced Study, 1 Einstein Dr., Princeton NJ 08540, USA}
\email{camillo.delellis@ias.edu}
	
	\begin{abstract}
    We give a rigorous mathematical treatment of some portions of the theory developed by Alexander Migdal on the momentum loop equation. 
	\end{abstract}

    \maketitle
	
	\tableofcontents
	
	\section{Introduction}
	
    We consider the Cauchy problem for the Navier-Stokes equations in $\R^3$
	\begin{equation}\label{NS}
		\begin{cases}
			\partial_t u + (u \cdot \nabla) u - \nu \Delta u + \nabla p = 0, \\
			\div u = 0, \\
            u (\cdot, 0) = u_0
		\end{cases}
		\tag{NS}
	\end{equation}
	where $u: \mathbb{R}^3 \times [0,T] \to \mathbb{R}^3$ is the velocity of the fluid flow, $p: \mathbb{R}^3 \times [0,T] \to \mathbb{R}$ is the pressure field, $\nu>0$ is the kinematic viscosity, and $u_0$ a soleinodal initial data. Here, the time interval $[0,T]$ is defined with $T > 0$ and may extend to $T = +\infty$. We will consider either classical (regular) solutions, for which the global existence in time is however a famous long-standing open problem, or the most restrictive type of weak solutions for which global existence is known, which are called ``suitable weak solutions'' in the literature after the pioneering works of Scheffer and Caffarelli, Kohn, and Nirenberg, cf. \cite{Scheffer76,Scheffer77} and \cite{caffarelli1982}. 

    \medskip

    In his works \cite{Migdal23,Migdal24,Migdal25,Migdalnuovo}, Migdal approaches the theory of decay turbulence introducing a rather innovative and interesting framework to study, roughly speaking, statistical properties of solutions of \eqref{NS}. The aim of this note is twofold: translate Migdal's work into a language which is more familiar to pure mathematicians and provide rigorous mathematical arguments for some statements. 
 
    \subsection{The momentum loop equation and the Euler ensemble} The starting point of Migdal's theory is to consider a probability measure $\mu_0$ on the space of initial data $u_0$ and ``evolve'' it according to \eqref{NS}. Ideally, we would like to have a suitable space $X$ of initial data for which the solutions of \eqref{NS} form a Lipschitz semigroup: we could then take $\mu_0$ to be concentrated on $X$ and we could push it forward over the trajectories of Navier-Stokes using the semigroup. This would define an evolving distribution of probability measures $\mu_t$, which we could interpret as a ``probabilistic solution'' of the Cauchy problem. A typical example of $\mu_0$ considered by Migdal is some Gaussian probability measure centered over a particular deterministic initial data $u_0$ modeling (small) fluctuations around $u_0$ (cf. Definition \ref{def:GaussInitialVel} for a definition in the case $u_0=0$): $\mu_t$ can be thought as evolving the fluctuations around the deterministic solution. Note, however, that unlike many other ``stochastic versions'' of Navier-Stokes considered in the mathematical literature, there is no forcing term in this picture. 
    
    Since the global uniqueness of solutions is not known (and quite likely linked to the finite-time blow-up problem for smooth solutions), the ideal scenario outlined above poses some serious challenges, in particular concerning the uniqueness of $\mu_t$. However, it is at least possible to construct a probability measure $\mu_t$ which is concentrated on ``weak trajectories'' of the Navier-Stokes equations,  cf. Theorem \ref{t:prob} in Section \ref{s:probabilistic-solutions}. This is not a new remark, as it was already noted in the seventies by a few authors, among them we point out the paper \cite{LadyVers76}, which seems to be the first to consider such setting in the framework of Leray-Hopf weak solutions. We propose in this paper a proof of the existence of $\mu_t$ which, unlike that of \cite{LadyVers76}, just uses soft arguments and the classical theory of existence of weak solutions. The argument is rather flexible and can easily incorporate extra conditions, in particular it delivers a probability measure $\mu_t$ which is concentrated on suitable weak solutions. We refer the interested reader to \cite{Hopf52,Foias72,Foias73,VishikFursikov78,VishikFursikov88,NPS13} for further works on statistical solutions to the Navier-Stokes equations with random initial data.

    \medskip

    The second step of Migdal's framework is to fix a (sufficiently regular) loop $C$ in $\mathbb R^3$ and compute a suitable imaginary exponential of the circulation of the solutions $u (\cdot, t)$ on the loop. 
    At this point Migdal translates the Navier-Stokes flow for the solution of \eqref{NS} into an equation for the loop functional in the infinite-dimensional space of all loops. The main tool to do so is using the so-called ``area-derivative'': roughly speaking the latter allows to recover the various terms in the equation for the loop functional as some sort of differential operators in the loop space applied to the loop functional. Migdal considers two possible approaches: in \cite{Migdal23,Migdal24,Migdal24b} the loop $C$ is fixed in space, while in \cite{Migdalnuovo} the loop is transported along the flow of the Navier--Stokes solution. These two perspectives naturally lead to different functionals and thus to different equations. Following Migdal's terminology, we will refer to the latter as the ``liquid loop approach'' (see the discussion in Section \ref{s:loop-equation}).

    \medskip
    
    Clearly, evaluating the circulation of $u (\cdot, t)$ over a suitably regular loop is a well-defined operation if $u (\cdot, t)$ is smooth enough, while it poses some issues when $u (\cdot, t)$ is merely a weak solution. We show in this paper that, however, suitable weak solutions of Navier-Stokes are regular enough to make sense of the circulation over loops in a good way. Moreover, if the probabilistic solution is invariant under translation (an assumption satisfied by many of the distributions $\mu_0$ of initial data considered by Migdal), then Migdal's loop equation does make sense for the $\mu_t$ constructed in this paper.  This allows us to give a ``weak sense'' to the loop equation under the latter assumption, cf. Corollary \ref{c:loop-equation}. In the case of the liquid loop equation we run into serious regularity issues for the advection of the loop, cf. Section \ref{s:liquid-functional}, and we therefore just report the ``formal computation'', cf. Section \ref{s:liquid-equation}.

    \medskip

    The next step of Migdal's theory is to introduce a suitable ``infinite-dimensional Fourier transform'' of the random variable induced by $\mu_t$: in Migdal's terminology the transformed variables are the ``momenta'' (in analogy with how momentum is defined in classical quantum mechanics) and they satisfy a transformed equation called ``momentum loop equation''.  Migdal then finds a particular family of explicit solutions of the momentum loop equation, which he calls ``Euler ensemble'': this family is what he puts forward as describing decaying turbulence, and can be thought as an ergodic hypothesis of sort.

    \subsection{Discretizations} In this paper we point out a key difficulty that a rigorous mathematical formulation of the theory encounters at the introduction of the momentum variables: while the starting object $\mu_t$ is a probability measure in physical space, we can prove that, under some minimal assumptions, the corresponding Migdal's momentum cannot be a nonnegative measure on any reasonable space of generalized functions, unless $\mu_t$ is ``trivial'' in a suitable sense (cf. Theorem \ref{t:non-existence}). In particular a rigorous definition of Migdal's momentum in this infinite-dimensional setting seems to require a theory of generalized functions over an infinite-dimensional space.

    A possible way around this obstacle is to suitably discretize the problem, as suggested by Migdal himself in \cite{Migdal24}. More precisely, we can fix a large number $N$ and look at polygonal loops with $N$ vertices. This (sub)space is now finite-dimensional and the loop functional computed on these loops induces, through $\mu_t$, a random variable over it. We can therefore take a classical Fourier transform of this random variable, gaining a ``discretized Migdal momentum'', cf. Section \ref{subsec:discretizedmomentum} (the latter is typically not a probability measure either, however it is a well-defined object). In the discretized setting of polygonal loops it is possible to define differential operators which approximate the operators used by Migdal in the infinite-dimensional loop space. In fact the work \cite{Migdal24} proposes a suitable discretization, studies its Fourier transform, and then derives a corresponding special family of solutions which, in the limit as $N\uparrow \infty$, recovers the Euler ensemble. Following Migdal, we will refer to these families (depending on the parameter $N$) as ``discretized Euler ensembles''.

    \medskip

    As it is typical in discretizations of infinite-dimensional problems, there is a variety of choices on how to perform them. In this paper we show that there is at least one choice for which, if we assume enough regularity on the Navier-Stokes trajectories over which $\mu$ is concentrated, we can prove the existence of an approximate equation satisfied by the discretized loop functional. In particular the approximate equation is satisfied up to error terms which can be bounded with suitable positive powers of $\frac{1}{N}$ and suitable norms of the solution, cf. Theorem \ref{t:loop-equation-discretized-2}. A technical difference with Migdal's computations is that we use use an appropriate truncation of the Biot-Savart operator to be able to define certain objects. This, however does not affect the overall results. A more substantial difference with Migdal's original computations in \cite{Migdal24} is the presence of a vanishing normalizing logarithmic factor in front of one of the approximate operators, the velocity operator in Section \ref{subsec:vel_op}, see Theorem \ref{t:loop-equation-discretized-2}. This correction was suggested to us by Migdal after we discussed our computations and it appeared first in his subsequent work \cite{Migdal24b}. In this paper we show that the correction must be taken into account even in Migdal's approximation, cf. \eqref{eq:vel_op Migdal} and Lemma \ref{lemma:vel_op Migdal}. Migdal in \cite{Migdal24} shows that the correction does not affect the computations on the discretized Euler ensemble and we can confirm this, cf. Section \ref{s:discussione}. 
    
    The remaining difference, which seems of technical nature, has already been underlined and is in the choice of the discretized operators. With the choice proposed in this paper it is easier to control the error terms in the approximate equation. If we follow the discretization proposed by \cite{Migdal24} we can still rigorously derive that all but one of the discretized operators used there behave as expected in \cite{Migdal24}. We are in fact able to bound the error terms arising from the discretized vorticity, velocity, and diffusion operators as defined by Migdal, cf. Lemma \ref{l:exponential-computation-2}, \ref{l:diffusion migdal} , and \ref{lemma:vel_op Migdal}. His definition of the discretized advection operator generates instead some error terms which, at present, we are unable to control, cf. Sections \ref{s:advection}. In the latter case our investigations are so far inconclusive: while we are unable to control these error terms, we are not able to show that they are large either.

    \medskip

    In \cite{Migdalnuovo} Migdal proposes a different derivation of the discretized Euler ensemble, using instead the liquid loop equation. In this derivation he avoids the advection term in the corresponding equation for the Euler ensemble and he asserts that the discretization would then rely solely on the vorticity and diffusion operators. It can be easily checked, cf. Section \ref{s:discussione} that the discretized Euler ensemble then also satisfies the corresponding discretized equation. This suggests therefore two interesting further directions: 
    \begin{itemize}
        \item solve the technical difficulties and handle the more intricate error terms arising from Migdal's descretized advection operator and derive his discrete approximation of the momentum loop equation;
        \item revert to a discretization of the liquid loop equation and possibly avoid the problematic error terms all together.
    \end{itemize}

\subsection{Notation}
Throughout the manuscript, we write $\const$ for a universal constant, 
whose value may change from line to line. 
If a constant depends only on certain parameters $a,b, \ldots $, 
we indicate this by writing $\const(a,b, \ldots)$.

\subsection{Acknowledgments} Obviously so much of this paper owes to our numerous conversations with Sasha Migdal and his constant interaction with us in the last couple of years. 

Part of this work was carried out during a visit of EB to the IAS, he gratefully acknowledges the support of the National Science Foundation under Grant No. DMS-2424441. The research of CDeL has been supported by the National Science Foundation through the DMS-2350252 grant and by the Simons Initative on the Geometry of Flows (Grant Award ID BD-Targeted-00017375, CDeL).

\section{Probability measures on trajectories of Navier-Stokes}\label{s:probabilistic-solutions}

Throughout this section, we let $\Omega$ denote either $\mathbb{R}^3$ or the torus $\mathbb{T}^3 := \mathbb{R}^3/\mathbb{Z}^3$. We denote by $L^2_\sigma(\Omega)$ the space of square-integrable, divergence-free vector fields on $\Omega$. For simplicity, we will often omit the dependence on $\Omega$ and write $L^2_\sigma$. We will give all our arguments in the case of $\mathbb R^3$, as the ones on the periodic torus are minor modifications (and in fact are typically easier given that the only harmonic functions on the torus are the constants). 

\medskip
We are interested in global solutions to \eqref{NS} with a random initial velocity field $u_0 \in L^2_\sigma$, sampled from a Borel probability measure $\mu_0 \in \mathcal{P}(L^2_\sigma)$. Since $L^2_\sigma$ is separable, the $\sigma$-algebra generated by the open sets of the weak topology coincides with that generated by the strong topology, so there is no ambiguity in the notion of Borel measurability. 

The main result of this section is that, for any such initial distribution, one can construct a global \textit{random admissible solution} to \eqref{NS}. Before stating the main theorem, we introduce some notation. 

\begin{definition}\label{d:space}
	Let $I$ be an interval (possibly finite or infinite, open or closed). We define $C_w(I, L^2(\Omega))$ as the space of measurable functions $u : \Omega \times I \to \mathbb{R}^3$ which belong to $L^\infty (I, L^2_\sigma (\Omega))$ and such that $t\mapsto u (\cdot, t)$ is continuous in the weak topology of $L^2$.
	Functions are identified if they differ only on a set of (space-time) measure zero.
\end{definition}

The following is a standard characterization of $C_w (I, L^2_\sigma (\Omega))$, whose simple proof is left as an exercise for the reader.

\begin{lemma}\label{l:weak-continuity}
	A function $u \in C_w(I, L^2_\sigma)$ if and only if the following two properties hold:
    \begin{itemize}
    \item[(i)] The map $t \mapsto \|u(\cdot, t)\|_{L^2}$ belongs to $L^\infty(I)$;
    \item[(ii)] for every divergence-free test function $\varphi \in C^\infty_c(\Omega)$, the map
	\begin{equation}\label{e:def-Phi}
		t \mapsto \Phi(t) := \int_\Omega \varphi(x) \cdot u(x,t)\, dx
	\end{equation}
	agrees almost everywhere with a continuous function on $I$.
    \end{itemize}
	Moreover, for every $t_0 \in I$, the time average
	\[
	r \mapsto \fint_{I \cap (t_0 - r, t_0 + r)} u(x,s)\, ds
	\]
	converges weakly in $L^2$ as $r \to 0$. This defines a well-posed trace of $u$ on $\Omega \times \{t\}$, which coincides with the original $u$ for almost every $(x,t)$. Furthermore, if this trace is used in the pointwise definition of $\Phi$ in~\eqref{e:def-Phi}, then $\Phi$ is continuous at every $t \in I$.
\end{lemma}

Thanks to Lemma~\ref{l:weak-continuity}, we can define the evaluation map at time $t$ on the space $C_w(I, L^2)$.

\begin{definition}
	For any $t \in I$, the \emph{evaluation map}
	\[
	e_t : C_w(I, L^2) \to L^2
	\]
	is defined by $e_t u := u(\cdot, t)$, where the right-hand side denotes the trace of $u$ at time $t$ as defined in Lemma~\ref{l:weak-continuity}.
\end{definition}

\begin{theorem}[Probabilistic Solutions]\label{t:prob}
	Let $\mu_0$ be a Borel probability measure on $L^2_\sigma(\Omega)$. Then there exists a Borel probability measure $\mu$ on the space $C_w(\mathbb{R}_+, L^2_\sigma)$ such that:
	\begin{itemize}
		\item[(a)] $\mu$ is supported on the set of admissible weak solutions to \eqref{NS} with $T = \infty$ (see Definition~\ref{d:admissible});
		
		\item[(b)] the map $t \mapsto \mu_t := (e_t)_\sharp \mu$ is continuous in the weak$^*$ topology;
		
		\item[(c)] $(e_0)_\sharp \mu = \mu_0$.
	\end{itemize}
\end{theorem}

Let $I := [0, T]$. A probability measure $\mu$ on the space $C_w(I, L^2_\sigma)$ satisfying properties (a) and (b) in Theorem~\ref{t:prob} will be called a \emph{probabilistic admissible weak solution} to \eqref{NS} on the interval $I$.

\begin{remark}[Restriction of Probabilistic Solutions]
	If $I \subset \mathbb{R}_+$ is an interval, we denote by
	\[
	R_I : C_w(\mathbb{R}_+, L^2) \to C_w(I, L^2)
	\]
	the restriction map. If $\mu$ is a probabilistic admissible weak solution to \eqref{NS} as constructed in Theorem~\ref{t:prob}, then $(R_I)_\sharp \mu$ is a probabilistic admissible weak solution on the interval $I$.
\end{remark}

		\begin{definition}\label{d:admissible}
		For $T\in (0, \infty]$ the space $\mathcal{A}_T$ of admissible weak solutions of Navier-Stokes consists of the vector fields $u\in C_w ([0,T), L^2)$ satisfying the following properties:
		\begin{itemize}
			\item[(i)] $Du \in L^2 (\mathbb R^3\times I)$ (resp. $L^2 (\mathbb T^3\times I)$;
			\item[(ii)] Let $p (\cdot, t)= \frac{1}{4\pi|\cdot|} * \partial^2_{ij} ((u^i u^j) (\cdot, t)$ (following Einstein conventions) when the spatial domain is $\mathbb R^3$, or $p (\cdot, t)$ be the unique average-free function such that $-\Delta p= \partial^2_{ij} (u^i u^k) (\cdot, t)$ when the domain is $\mathbb T^3$; then $(u,p)$ is a distributional solution of 
			\begin{equation}\label{e:NS}
				\left\{
				\begin{array}{l}
					\partial_t u + (u\cdot \nabla) u + \nabla p = \Delta u \\ \\
					{\rm div}\, u =0
				\end{array}
				\right.
			\end{equation}
			and
			\begin{equation}\label{e:energy}
				(\partial_t -\Delta) \frac{|u|^2}{2} + |Du|^2 + {\rm div}\, \left(\left(\frac{|u|^2}{2} + p \right)\right) \leq 0\, ;
			\end{equation}
			\item[(iii)] The following energy estimate holds for all $t\in [0,T)$:
			\begin{equation}\label{e:energy-2}
				\|u (\cdot, t)\|_{L^2}^2 + 2\int_0^t \|Du (\cdot, s)\|_{L^2}^2 ds \leq \|u_0\|_{L^2}^2\, .
			\end{equation}
		\end{itemize}
	\end{definition}
	
	We recall that, because of (a) and $u\in C_w (I, L^2_\sigma)$, we know that $u\in L^{\frac{10}{3}}$, and from the equation for the pressure we know that $p\in L^{\frac{5}{3}}$. In particular all the products appearing in \eqref{e:energy} are well defined as locally summable functions. The following is a well-known corollary of Leray's foundational work; we refer the reader, for instance, to the arguments in the Appendix of \cite{caffarelli1982}.
	
	\begin{lemma}\label{l:compact}
		If $\{u_k\} \subset \mathcal{A}_T$ is a sequence for which $u_k(\cdot, 0)$ converges strongly in $L^2$ and $T<\infty$, then a subsequence, not relabeled, converges in the topology of $C_w (I, L^2)$ to a $u\in \mathcal{A}_T$. Moreover $u_k\to u$ locally strongly in $L^q$ for any $q<\frac{10}{3}$. 
	\end{lemma}

	\subsection{Proof of Theorem \ref{t:prob}} For the sake of simplicity we assume $\Omega = \mathbb R^3$ and leave to the reader the obvious modifications in the case of $\Omega = \mathbb T^3$.
    
		For every $\varepsilon>0$, we consider Leray's regularization of the Navier-Stokes equation given by
		\begin{equation}\label{e:Leray}
			\left\{
			\begin{array}{l}
				\partial_t u + (u*\varphi_\varepsilon\cdot \nabla) u + \nabla p = \Delta u \\ \\
				{\rm div}\, u =0
			\end{array}
			\right.
		\end{equation}
		where $\varphi$ is a standard mollifier in space, $\varphi_\varepsilon (x) = \varepsilon^{-3} \varphi (\frac{x}{\varepsilon})$ and 
		\[
		u* \varphi_\varepsilon (x,t) = \int u (y,t) \varphi_\varepsilon (x-y)\, dy
		\]
		It is well known that:
		\begin{itemize}
			\item There is a unique semigroup $\Phi_\varepsilon: L^2 \to C_w ([0,\infty), L^2)$ of solutions of \eqref{e:Leray}, namely for every divergence free initial data $u_0\in L^2$ the vector field $u = \Phi_\varepsilon (u_0)$ is the unique solution of \eqref{e:Leray} with $u (\cdot, 0) = u_0$ such that in addition 
            \[
            p (\cdot, t) = \frac{1}{4\pi |\cdot|} * \partial^2_{ij} (u^i*\varphi_\varepsilon (\cdot, t) u^j (\cdot, t))
            \]
            (where we use Einstein's summation convention on repeated indices). This is in particular proved by Leray himself in \cite{Leray34}.
			\item For $\varepsilon_k\downarrow 0$ and every $u_0^k \to u_0$ strongly there is a subsequence, not relabeled, such that $\Phi_{\varepsilon_k} (u_0^k)$ converges, in $C_w ([0, T], L^2)$ for every $T>0$, to an admissible weak solution $u\in \mathcal{A}_\infty$ with $u (\cdot, 0)=u_0$.
		\end{itemize}
		We consider the measures $(\Phi_\varepsilon)_\sharp \mu_0$ and we extract a subsequence $\varepsilon_k\downarrow 0$ such that the measures $\mu^k_N := (R_{[0,N]} \circ (\Phi_{\varepsilon_k})_\sharp \mu_0$ converge (for $k\uparrow \infty$) weakly$^*$ to a measure $\mu_N$ on $C_w ([0, N], L^2_\sigma)$ for every integer $N$. This can be easily done because the topology on bounded subsets of $C_w ([0, N], L^2_\sigma)$ can be metrized and the corresponding metric space is $\sigma$-compact. Since $(R_{[0,N]})_\sharp (R_{[0, N']})_\sharp (\Phi_{\varepsilon_k})_\sharp \mu_0 = (R_{[0,N]})_\sharp (\Phi_{\varepsilon_k})_\sharp \mu_0$ for every $N<N'$, we also get 
        the identity $(R_{[0, N]})_\sharp \mu_{N'} = \mu_N$.

        \medskip

        {\bf Step 1.} We start by checking that $\mu_N$ is a probability measure. First of all, because $L^2_\sigma$ is a Polish space, $\mu_0$ is tight, which means that for every $\varepsilon > 0$ we can find a strongly compact subset $K\subset L^2_\sigma$ with the property that $\mu_0 (K^c) < \varepsilon$. Next denote by $Z$ the closure, in the topology of $Y:= C_w ([0,N], L^2_\sigma)$, of 
        \[
        Z' := \bigcup_k R_{[0,N]} (\Phi_{\varepsilon_k} (K))\, . 
        \]
        We claim that $Z$ is in fact a compact subset of $Y$. Consider a sequence $\{v_j\}\subset Z$. First of all, because $K\subset L^2$ is strongly compact, it is in particular bounded and thus $M:= \max \{\|u_0\|_{L^2} : u_0 \in K\} < \infty$. In particular, for every element $v\in Z'$ we must have
        \[
        \max \|v (\cdot, t)\|_{L^2} \leq M\, 
        \]
        by standard energy estimate, 
        and of course the latter bound is valid for every element $v\in Z$ as well. Hence the topology of $Y$ on $Z$ is metrizable and we denote by $d$ a metric which induces it. In particular, given a sequence $\{v_j\}\subset Z$, we can find a corresponding sequence $\{u_j\}\subset Z'$ with $d(u_j, v_j) \to 0$. In order to show compactness of $Z$ we just need to show that a subsequence of $\{u_j\}$, not relabeled, converges in the metric $d$ to some $u$. However, we must have $u_j = R_{[0,N]} (\Phi_{\varepsilon_j} (u_0^j))$ for some sequence of initial data $\{u_0^j\}\subset Z$. In particular we can assume that $u_0^j$ converges strongly to some element $u_0 \in K$. We now have two possibilities: either $\varepsilon_j\downarrow 0$, in which case we can apply Lemma \ref{l:compact} to extract a convergent subsequence from $\{u_j\}$, or infinitely many of the $\varepsilon_j$ are the same number $\varepsilon >0$, and we can use the continuity of the semigroup $\Phi_\varepsilon$ to conclude that $u_j$ converges to $R_{[0,N]} (\Phi_\varepsilon (u_0))$.

        We can now use Urysohn's Lemma to find a continuous function $\psi \in C_c (Y)$ with the property that 
        $\varphi \geq \mathbf{1}_Z$. But then we can estimate 
        \[
        \mu_N (Y) \geq \int \psi\, d\mu_N = \lim_{k\to \infty} \psi\, d (R_{[0,N]})_\sharp 
        (\Phi_{\varepsilon_k})_\sharp \mu_0 \geq \mu_0 (K) \geq 1-\varepsilon\, ,
        \]
        which implies $\mu_N (Y) \geq 1-\varepsilon$. Since $\varepsilon$ is arbitrary, this implies $\mu_N (Y) \geq 1$. On the other hand, because $\mu_0$ is a probability measure, we also have 
        \[
        \int \psi\, d\mu_N \leq 1
        \]
        for every nonnegative $\psi \in C_c (Y)$ with $\psi\leq 1$. In particular we conclude also that $\mu_N (Y) \leq 1$.
        
        \medskip

        {\bf Step 2.} We next claim that there is a unique probability measure $\mu$ on $C_w (\mathbb R^+, L^2_\sigma)$ with the property that $\mu_N = (R_{[0,N]})_\sharp \mu$. It is easy to define $\mu$ on ``cylinder'' sets of the form $A' = (R_{[0,N]})^{-1} (A)$, where $A\subset C_w ([0,T], L^2)$ are Borel subsets. In fact on such sets we can simply set 
        \[
        \mu (A') := \mu_N (A)
        \]
        and the definition is coherent because of the property $(R_{[0,N]})_\sharp)\mu_{N'} = \mu_N$ for $N'>N$. On the other hand such cylinder sets generate the Borel $\sigma$-algebra of $C_w (\mathbb R^+, L^2)$, and thus the existence and uniqueness of $\mu$ is a standard consequence of Caratheodory's extension theorem.

        \medskip
        
         {\bf Step 3.} We finally wish to show that the probability $\mu$ is concentrated on the set of admissible weak solutions. First of all, we observe that it suffices to show that the measure $\mu_N$ is concentrated on the set of admissible weak solutions on $\mathbb R^3\times [0,N)$, which we denote by $\mathcal{A}_T$. We start by fixing $\varepsilon >0$ and choosing the compact set $K$ as in Step 1. We then define $Z'_m$ to be
         \[
         Z'_m := \bigcup_{k\geq m} R_{[0,N]} (\Phi_{\varepsilon_k} (K))\, 
         \]
         and we let $Z_m$ be its closure in the topology of $Y= C_w ([0, N], L^2)$. Observe that 
         \[
         (\Phi_{\varepsilon_k})_\sharp \mu_0 (Z_m) \geq 1-\varepsilon \qquad \forall k \geq m\, .
         \]
         In particular, using the properties of weak$^*$ convergence of measures and the fact that $Z_m$ is compact, we have $\mu_N (Z_m)\geq 1-\varepsilon$ (the reader can supply the argument using Urysohn's functions with supports closer and closer to $Z_m$). Consider now the intersection
         \[
         W:= \bigcap_m Z_m\, .
         \]
         Note that $\mu_N (W) \geq 1-\varepsilon$. We claim that any element in $W$ is an admissible weak solution. In fact, by construction, for every fixed element $w\in W$ and every integer $m$, we can find an element in $Z'_m$ with distance $d$ at most $\frac{1}{m}$ from $w$. This means that there exists an integer $k(m)\geq m$ and an initial data $w_{0,m}\in K$ with the property that $w_m := \Phi_{\varepsilon_{k(m)}} (w_{0,m})$ satisfies $d (w, w_m)\leq \frac{1}{m}$. But, letting $m\uparrow \infty$, using the compactness of $K$ and leveraging Lemma \ref{l:compact}, we then conclude that $w$ is in fact an admissible weak solution of \eqref{NS} with some initial data $w_0\in K$. 

         In particular the above argument just shows that $\mu_N (\mathcal{A}_N) \geq 1- \varepsilon$. The arbitrariness of $\varepsilon$ implies that $\mu_N (\mathcal{A}_N) = 1$, and since $\mu_N$ is a probability measure, we have just shown that the complement of $\mathcal{A}_N$ has $\mu_N$ measure zero in $C_w ([0, N], L^2_\sigma)$. 

\subsection{Gaussian Measures on \texorpdfstring{$\mathbb{T}^n$}{Tn}}\label{subsec:NSrandomGauss}

In this section, we recall the construction of a Gaussian random initial velocity field, which appears frequently in Migdal's work.

Let $(\mathcal{S}, \mathcal{F}, \mathbb{P})$ be a probability space, and let $W \colon \mathcal{S} \to \mathcal{D}'(\mathbb{T}^n)$ be a centered Gaussian random distribution. By this we mean that, for every $f \in C^\infty(\mathbb{T}^n)$, the pairing $\langle W,f\rangle$ is a centered Gaussian random variable, and its covariance is given by
\begin{equation}\label{eq:covarianceW}
    \mathbb{E}\bigl[\langle W,f\rangle \langle W,g\rangle\bigr]
    = \langle f,g\rangle_{L^2}, \qquad f,g \in C^\infty(\mathbb{T}^n).
\end{equation}

One may construct such a random distribution by randomizing the Fourier basis as
\[
W(s) := \sum_{k \in \mathbb{Z}^n} \gamma_k(s)e^{ik\cdot x} \in \mathcal{D}'(\mathbb{T}^n), \qquad s \in \mathcal{S}, 
\]
where $\{\gamma_k\}_{k >0}$ are independent complex Gaussian random variables with mean zero and variance one, and we set $\gamma_0=0$ and $\gamma_{-k}:=\overline{\gamma_k}$ for every $k>0$.
\begin{remark}
    Note that $W(s)$ does not belong to $L^2_\sigma$ almost surely, since its Fourier coefficients decay too slowly. On the other hand, it is still more regular than a general distribution: $W(s)$ belongs almost surely to a negative Sobolev space. More precisely,
    \begin{equation}
        \mathbb{E}\bigl[\|W\|_{H^{-N}}^2\bigr] = \sum_{k \in \mathbb{Z}} \frac{1}{|k|^{2N}} < \infty
    \end{equation}
    whenever $N > n/2$.
\end{remark}
We then define the random vector field $\widetilde{W} := (W_1,\ldots,W_n)$, where $W_1,\ldots,W_n$ are independent copies of the scalar white noise $W$ described above.

\begin{definition}[Gaussian Initial Velocity]\label{def:GaussInitialVel}
	Fix a scale parameter $r_0 > 0$ and define the smoothed, divergence-free random field
	\begin{equation}\label{eq:xi}
		\xi(s) := P_H \, e^{\frac{r_0^2}{2} \Delta} \widetilde{W}(s) \in L^2_\sigma(\mathbb{T}^n),
        \quad s\in \mathcal{S},
	\end{equation}
	where $e^{t\Delta}$ denotes the heat semigroup and $P_H : L^2 \to L^2_\sigma$ is the Leray projection onto divergence-free vector fields.
\end{definition}

\begin{proposition}\label{p:Gaussian}
	Fix $r_0 > 0$, and let $\xi(s)$ be defined as in~\eqref{eq:xi}. Then $\mu_0:=\xi_\sharp \mathbb{P}$ is a centered Gaussian measure on $L^2_\sigma(\mathbb{T}^n)$, concentrated on the space $\bigcap_{N \ge 0} H^N(\mathbb{T}^n)$. Its covariance operator is given by
	\begin{equation}\label{eq:covariance}
		\mathbb{E} \left[ \langle f, \xi \rangle_{L^2} \,  \langle g, \xi \rangle_{L^2}  \right]
		= \langle f, e^{r_0^2 \Delta} g \rangle_{L^2},
		\quad f, g \in L^2_\sigma.
	\end{equation}
\end{proposition}

\begin{proof}
	The covariance formula \eqref{eq:covariance} follows directly from \eqref{eq:covarianceW} and the fact that the heat semigroup is self-adjoint in $L^2$.
	
	To verify regularity, we use that $P_H \colon H^N \to H^N$ is bounded for every $N \ge 0$, together with the explicit expression of $\widetilde{W}$:
\begin{equation}
\begin{split}
    \mathbb{E}\left[\left\|P_H e^{\frac{r_0^2}{2}\Delta}\widetilde{W}\right\|_{H^N}^2\right]
    &\le C(n,N)\,\mathbb{E}\left[\left\|e^{\frac{r_0^2}{2}\Delta}\widetilde{W}\right\|_{H^N}^2\right] \\
    &= C(n,N)\,\mathbb{E}\left[\sum_{i=1}^n \sum_{k\in\mathbb{Z}^n} |\gamma_{k,i}|^2 |k|^{2N} e^{-r_0^2|k|^2}\right] \\
    &\le C(r_0,n,N).
\end{split}
\end{equation}
\end{proof}

The law of $\xi(s)$ is invariant under spatial translations, since its covariance \eqref{eq:covariance} is translation invariant and Gaussian measures are fully determined by their covariance. Specifically, if $\tau_h : \mathbb{T}^n \to \mathbb{T}^n$ denotes the translation operator $x \mapsto x + h$, then
$\xi_\# \mathbb{P} = (\tau_h \circ \xi)_\# \mathbb{P}$.

The following proposition is a direct consequence of the translation invariance of the Navier–Stokes equations and the construction in Theorem~\ref{t:prob}.

\begin{proposition}\label{p:trans-invariance}
	Let $\xi(s)$ be defined as in \eqref{eq:xi}, and set $\mu_0 := \xi_\#\mathbb{P}$. Let $\mu$ be the probabilistic suitable weak solution to \eqref{NS} constructed in Theorem~\ref{t:prob} with $\Omega =\mathbb{T}^3$. Then $\mu$ is translation invariant, i.e.,
	\[
	(\tau_h)_\# \mu = \mu \quad \text{for every } h \in \mathbb{T}^3.
	\]
\end{proposition}

\begin{proof}
Since Leray's regularization \eqref{e:Leray} is invariant under translations, the semigroup $\Phi_\varepsilon$ is also invariant, in the sense that $\tau_h (\Phi_\varepsilon (u_0)) = \Phi_\varepsilon (\tau_h (u_0))$. In particular the probability measures $(\Phi_\varepsilon)_\sharp \mu_0$ are invariant as well using $(\tau_h)_\# (\Phi_\varepsilon)_\# \mu_0 = (\Phi_\varepsilon)_\#(\tau_h)_\# \mu_0 = (\Phi_\varepsilon)_\# \mu_0$. Given that $\mu$ is constructed in Theorem~\ref{t:prob} as a weak$^*$ limit of $(\Phi_{\varepsilon_k})_\# \mu_0$ for some suitable sequence $\varepsilon_k\downarrow 0$, the claim of the proposition follows at once. 
\end{proof}

\subsection{Gaussian Measures on \texorpdfstring{$\mathbb{R}^n$}{Rn}}\label{subsec:NSrandomGauss-Rn}

In this section, we consider an analogue of the construction of the Gaussian random variable $\xi(s)$ in $L^2_\sigma(\mathbb{T}^n)$ in the setting of divergence-free velocity fields on $\mathbb{R}^n$. In this case, the corresponding translation-invariant random field does not belong to $L^2_\sigma(\mathbb{R}^n)$, since it does not decay sufficiently fast at infinity.
\medskip

Let $(\mathcal{S}, \mathcal{F}, \mathbb{P})$ be a probability space. We denote by $S(\mathbb{R}^n;\mathbb{R}^n)$ the space of Schwartz vector fields, by $S' (\mathbb{R}^n, \mathbb{R}^n)$ its dual, namely the space of tempered distribution, and recall that $P_H \colon L^2 \to L^2_\sigma$ is the Leray projector. We construct a centered Gaussian random distribution
\begin{equation}\label{eq:WRd}
 W \colon \mathcal{S} \to S'(\mathbb{R}^n;\mathbb{R}^n)   
\end{equation}
with covariance
\begin{equation}\label{eq:covarianceWRn}
    \mathbb{E}\bigl[\langle W,f\rangle \langle W,g\rangle\bigr]
    = \langle P_H f, P_H g\rangle_{L^2}, \qquad f,g \in S(\mathbb{R}^n;\mathbb{R}^n).
\end{equation}
 We recall that $W$ is said to be centered Gaussian if, for every $f \in S(\mathbb{R}^n;\mathbb{R}^n)$, the real-valued random variable $\langle W,f\rangle$ is centered Gaussian.

To construct $W$ with covariance \eqref{eq:covarianceWRn}, we apply a version of the Bochner--Minlos theorem, which we state here for the reader's convenience. We refer to \cite[Appendix A]{HOUZ10} for a detailed discussion and proof.

\begin{theorem}[Bochner--Minlos]\label{thm:Bochner-Mnlos}
    Let
    \[
    g \colon S(\mathbb{R}^n;\mathbb{R}^n) \to \mathbb{R}
    \]
    satisfy the following properties:
    \begin{enumerate}
        \item $g(0)=1$.
        \item $g$ is positive definite, i.e.
        \[
        \sum_{i,j=1}^N z_i \overline{z_j}\, g(f_i-f_j) \ge 0
        \]
        for every $N \in \mathbb{N}$, every $f_1,\dots,f_N \in S(\mathbb{R}^n;\mathbb{R}^n)$, and every $z_1,\dots,z_N \in \mathbb{C}$.
        \item $g$ is continuous with respect to the Fr\'echet topology of $S(\mathbb{R}^n;\mathbb{R}^n)$.
    \end{enumerate}
    Then there exists a Borel probability measure $\nu \in \mathcal{P}(S'(\mathbb{R}^n;\mathbb{R}^n))$ such that
    \begin{equation}\label{eq:identity}
        g(f) = \int_{S'(\mathbb{R}^n;\mathbb{R}^n)} e^{i \langle \omega, f\rangle}\, d\nu(\omega),
        \qquad \text{for every $f \in S(\mathbb{R}^n;\mathbb{R}^n)$}.
    \end{equation}
\end{theorem}

We apply Theorem \ref{thm:Bochner-Mnlos} to the functional
\begin{equation}
    g(f):= \exp\Bigl\{-\frac{1}{2} \| P_H f \|_{L^2}^2\Bigr\}.
\end{equation}
This functional clearly satisfies conditions (1)--(3). The measure $\nu$ then defines a random variable \eqref{eq:WRd} with law $\nu$. The identity \eqref{eq:identity}, with the above choice of $g$, gives the characteristic function of the real-valued random variable $\langle W,f\rangle$. It follows that $\langle W,f\rangle$ is a centered Gaussian random variable with variance $\|P_H f\|_{L^2}^2$. By polarization, we obtain \eqref{eq:covarianceWRn}.
\medskip

\begin{definition}[Gaussian Initial Velocity]\label{def:GaussInitialVel-Rn}
    Fix a scale parameter $r_0 > 0$ and define the smoothed random field
    \begin{equation}\label{eq:xi-Rn}
        \xi(s) := e^{\frac{r_0^2}{2}\Delta} W(s), \qquad s \in \Omega.
    \end{equation}
\end{definition}

For $\mathbb{P}$-almost every $s \in \Omega$, the field $\xi(s)$ belongs to $C^\infty(\mathbb{R}^n;\mathbb{R}^n)$, since the heat semigroup maps tempered distributions to smooth functions. Moreover, $\xi(s)$ is almost surely divergence-free. Indeed, for every $p \in C_c^\infty(\mathbb{R}^n)$,
\begin{equation}
    \langle \div\xi, p\rangle_{L^2}
    = - \langle \xi, \nabla p\rangle_{L^2}
    = - \Bigl\langle W, e^{\frac{r_0^2}{2}\Delta}\nabla p \Bigr\rangle_{L^2},
\end{equation}
and therefore
\begin{equation}
    \mathbb{E}\Bigl[\bigl|\langle \div\xi, p\rangle_{L^2}\bigr|^2\Bigr]
    = \Bigl\| P_H \Bigl(e^{\frac{r_0^2}{2}\Delta}\nabla p\Bigr)\Bigr\|_{L^2}^2
    = 0.
\end{equation}
It is immediate to compute the covariance of $\xi$. For every $f,g \in L^2_\sigma$,
\begin{equation}\label{eq:covariance-Rn}
    \mathbb{E}\bigl[\langle f,\xi\rangle_{L^2}\,\langle g,\xi\rangle_{L^2}\bigr]
    = \langle f, e^{r_0^2 \Delta} g\rangle_{L^2}.
\end{equation}
We next observe that the law of $\xi(s)$ is invariant under spatial translations. Indeed, the covariance \eqref{eq:covariance-Rn} is translation invariant, and a centered Gaussian measure is uniquely determined by its covariance. More precisely, if $\tau_h$ denotes the translation operator
\begin{equation}
    (\tau_h u)(x) := u(x+h),
\end{equation}
then $\xi_\#\mathbb{P} = (\tau_h \circ \xi)_\# \mathbb{P}$ for every $h \in \mathbb{R}^n$.

On the other hand, $\xi(s)$ does not belong to $L^2(\mathbb{R}^n;\mathbb{R}^n)$ almost surely, since it does not decay at infinity. This may be seen either by computing the expected $L^2$ norm, which is infinite, or by observing that there are no nontrivial translation-invariant random variables with values in $L^2_\sigma(\mathbb{R}^n)$. It is an interesting question whether one can construct random weak solutions which satisfy the local energy inequality and take the initial condition $\xi(s)$. We do not address this question here, but we point out that it might be a subtle one, as the global-in-time existence of such weak solutions for deterministic data is not known in the class of uniformly local $L^2$ initial data, cf. \cite{BKT}.

\section{Loop Functional}\label{s:loop-functional}

We study the \emph{loop functional} associated with $\mu$, where $\mu$ is a probabilistically suitable weak solution of the Navier--Stokes equations \eqref{NS} with viscosity $\nu>0$ and initial datum $\mu_0$ supported on $L^2_\sigma(\Omega)$, with $\Omega$ equal to either $\mathbb{R}^3$ or $\mathbb{T}^d$.
A key example to keep in mind is the Gaussian initial datum $\mu_0=\xi_\sharp \mathbb{P}$ in the case $\Omega=\mathbb{T}^d$ (see Section~\ref{subsec:NSrandomGauss}).

We consider the circulation of a time-dependent velocity field $u(x,t)$ along a piecewise smooth loop $C : \mathbb S^1 \to \Omega$, defined by
\begin{equation}
	\Gamma[u,C,t] := \int_C u(x,t)\cdot dx 
	= \int_0^1 u(C(\theta),t) \cdot C'(\theta)\, d\theta.
\end{equation}

The free loop space over $\Omega$ is defined as
\begin{equation}
	\mathcal{L}\Omega := \{\, C : \mathbb S^1 \to \Omega \; : \; C \text{ piecewise smooth} \,\}.
\end{equation}

\begin{definition}[Loop Functional]\label{def:loopfunctional}
	Fix a parameter $\gamma > 0$. The \emph{loop functional} associated with $\mu$ is defined for all $t \ge 0$ and $C\in \mathcal{L}\Omega$ as
	\begin{equation}\label{eq:loopfunctional}
		\Psi[C,t] := \mathbb{E}_t \left[ \exp\left\{ \frac{i \gamma}{\nu} \Gamma[u, C,t] \right\} \right]
        = \int \left[ \exp\left\{ \frac{i \gamma}{\nu} \Gamma[u, C,t] \right\} \right]\, d\mu_t (u)\, .
	\end{equation}
	where the expectation is taken with respect to $\mu_t = (e_t)_\sharp \mu$.
\end{definition}
Recall that $\mu$ is a probability measure concentrated on $C_w(\mathbb{R}_+, L^2)$, and the time marginal $\mu_t := (e_t)_\sharp \mu$ is defined via the evaluation map at time $t$. More explicitly, we can write 
\begin{equation}
	\Psi[C,t] = \int \exp\left\lbrace \frac{i \gamma}{\nu} \int_0^1 u(C(\theta),t) \cdot C'(\theta)\, d\theta \right\rbrace \, d\mu_t (u).
\end{equation}
Note that the dependence of the functionals $\Gamma$ and $\Psi$ on $t$ is actually just induced by the dependence of the solution $u$ (and the probability measure $\mu_t$) on $t$. Nonetheless, since this dependence is important later on, we have chosen to highlight it with our notation.

\begin{remark}
When $t=0$ and $\mu_0$ is as in Proposition \ref{p:trans-invariance}, the loop functional can be computed explicitly using the Gaussian structure of the initial distribution. In the case $\Omega=\mathbb{R}^3$, the formula takes the particularly simple form
\begin{equation}
    \Psi[C,0]
    =
    \exp\left\{
    -\frac{1}{2}\frac{\gamma}{\nu}
    \int_0^1 \int_0^1
    \frac{
    \exp\left\{-\frac{|C(\theta_1)-C(\theta_2)|^2}{4r_0^2}\right\}
    }{(4\pi r_0^2)^{3/2}}
    C'(\theta_1)\cdot C'(\theta_2)\, d\theta_1 d\theta_2
    \right\}.
\end{equation}
\end{remark}

The next result ensures that the loop functional is well-defined for every piecewise smooth loop. This is not immediately evident from the explicit expression \eqref{eq:loopfunctional}, since suitable weak solutions are not known to be continuous in time at every $t \geq 0$, and the circulation $\Gamma[u,C,t]$ may not be defined pointwise.

\begin{proposition}\label{p:loop-regularity}
	Let $\mu$ be a probabilistic suitable weak solution of the Navier–Stokes equations \eqref{NS} in $\Omega$, with an initial distribution $\mu_0$ such that 
    \begin{equation}
    \int \|u\|_{L^2}^2\, d\mu_0 (u) < \infty\, .
    \end{equation}
    Then we have the estimate
\begin{equation}
    \int_0^T \int \Gamma[u,C,t]\, d\mu_t(u)\, dt 
    \le \const \int \|u\|_{L^2}^2\, d\mu_0(u)\, .
\end{equation}
In particular, 
\[
u(\cdot, t) \mapsto \exp\left\{ \frac{i \gamma}{\nu} \Gamma[u, C, t] \right\}
\]
belongs to $L^\infty(\mu_t \otimes dt)$.
Assume moreover that $\Omega=\mathbb{T}^3$ and $\mu_t$ is translation invariant in the sense of Proposition \ref{p:trans-invariance}. Then $t\mapsto \Psi[C,t]$ belongs to $W^{1,q}$ for every $q<2$ and we have the following representation formula for its time derivative
    \begin{align}\label{e:representation-formula}
\frac{d}{dt} \Psi [C, t]
    =  \frac{i\gamma}{\nu}
\int \int_0^1 \int_{\mathbb T^3} \phi(y) \partial_t u (C +y, t)\cdot C' \, e^{ \frac{i \gamma}{\nu} \Gamma [u, C+y, t]}\, 
\, dy\, d\theta\, d\mu_t (u)\, , 
    \end{align}
    for every smooth test function $\phi$ with $\int_{\mathbb{T}^3} \phi = 1$.
\end{proposition}

Note that the right hand side of \eqref{e:representation-formula} encodes the translation invariance in the loop $C$ of the left hand side.

\begin{proof}
	We proceed in two steps. First, we show that the circulation $\Gamma[u, C,t]$ is a well-defined function in $L^1_{\mathrm{loc}}(\mathbb{R}_+)$ with uniform estimates for any admissible weak solution. To this end, we employ refined bounds on suitable solutions to \eqref{NS}. In the second step, we apply an averaging argument based on the translation invariance of the loop functional, namely 
    \begin{equation}
        \Psi[C + h, t] = \Psi[C, t], \quad \text{every $h \in \mathbb{T}^3$ and $C\in \mathcal{L}\mathbb{T}^3$}
    \end{equation}
   to obtain time derivative estimates for $\Psi[C,t]$.

	\medskip
	\noindent
	{\bf Step 1.} Let $u(x,t)$ be suitable weak solution to \eqref{NS} in $\Omega$. For every piecewise smooth loop $C$, the circulation is bounded by
	\begin{equation}\label{e:stima-circolazione}
		\int_0^T |\Gamma[u, C,t]|\, dt \le \const T^{1/4} \|u_0\|_{L^2}^{1/2}(\| u_0 \|_{L^2}^2 + 1)^{3/4} |C|
	\end{equation}
	where $|C|$ denotes the length of the loop. It is indeed an estimate of Tartar that Leray's solutions belong to $u\in L^1_t C_x$. In fact a stronger estimate has been proven by Galeati in \cite[Corollary 4.3]{Galeati}
	\begin{equation}\label{e:Galeati}
	\int_0^T \| Du(\cdot, t) \|_{L^{3,1}}\, dt \le \const T^{1/4} \|u_0\|_{L^2}^{1/2}(\| u_0 \|_{L^2}^2 + 1)^{3/4}
	\end{equation}
    where $L^{3,1}$ denotes the corresponding Lorentz space, cf. \cite[Section 1.8]{Ziemer}. 
    From the Sobolev embedding for Lorentz spaces (see \cite[Section 2.9 \& Section 2.10]{Ziemer}) it follows that 
    \[
    \|u (\cdot, t)\|_{C^0} \leq \const (\|Du (\cdot, t)\|_{L^{3,1}} + \|u (\cdot, t)\|_{L^2}).
    \]
    Using \eqref{e:Galeati} and the energy estimate $\|u(\cdot, t)\|_{L^2}\leq \|u_0\|_{L^2}$, we immediately achieve \eqref{e:stima-circolazione}. 
\medskip

\noindent
{\bf Step 2.} Let $\mu$ as in the statement. For every $\phi\in C^\infty(\mathbb{T}^3)$ we claim the identity
\begin{equation}\label{eq:average1}
	\Psi[C,t] = \int\int_{\mathbb T^3} \phi(y)\exp\left\lbrace \frac{i \gamma}{\nu} \int_0^1 u(C(\theta) + y,t) \cdot C'(\theta)\, d\theta \right\rbrace\, dy \,  d\mu(u).
\end{equation}
for a.e. $t\ge 0$.	
From Proposition~\ref{p:trans-invariance}, we deduce that $\Psi[C + y, t] = \Psi[C, t]$ for every $y \in \mathbb{T}^3$. The claim then follows from the Fubini–Tonelli theorem and the estimate established in Step~1, which ensures that the function
\[
u(\cdot, t) \mapsto \exp\left\{ \frac{i \gamma}{\nu} \Gamma[u, C, t] \right\}
\]
belongs to $L^\infty(\mu_t \otimes dt)$.

\medskip
We can now prove that the map $t \mapsto \Psi[C,t]$ belongs to a Sobolev space. First of all we recall the estimates in \cite[Theorem 3.4]{Lions-book}: for every $q<2$ if we set 
\[
r = \frac{3q}{4q-2}
\]
Leray's solution of \eqref{e:NS} satisfy the estimate
\[
\|\partial_t u\|_{L^q_t L^r_x} \leq \const (T,q) \|u_0\|_{L^2}^2\, ,
\]
This allows us to take a weak derivative in time of \eqref{eq:average1} and prove the identity
\eqref{e:representation-formula}. In particular note that the inner integral in $y$ in \eqref{e:representation-formula} makes sense since $\phi$ is a smooth test function and $y\mapsto \exp (i \nu^{-1} \gamma \Gamma [u, C+y, t])$ is a bounded function.
\end{proof}

\begin{remark}[Deterministic Loop Functional]
In principle, one can consider the loop functional associated with a deterministic solution of the Navier--Stokes equations. Namely, one may consider a measure $\mu$ concentrated on a single Leray--Hopf (or classical) trajectory $u(x,t)$ and define
\begin{equation}
\Psi[C,t] := \exp\left\lbrace \frac{i \gamma}{\nu} \int_0^1 u(C(\theta),t) \cdot C'(\theta)\, d\theta \right\rbrace .
\end{equation}
Most of the results in this paper are pathwise, and therefore they also hold in the deterministic setting (see for instance Theorem \ref{t:loop-equation-discretized} which is stated directly in the deterministic setting). However, the stochastic framework has several advantages. From the physical standpoint, it is natural to include small random fluctuations around a given initial datum. From the mathematical point of view, randomizing the initial condition allows one to gain regularity in time (see Proposition~\ref{p:loop-regularity}) which is particularly convenient when dealing rigorously with possibly singular Leray solutions.
\end{remark}

\subsection{Liquid Loop Functional}\label{s:liquid-functional}
In \cite{Migdalnuovo} Migdal proposes a different functional: having fixed an initial (piecewise) smooth loop $C(\theta,0)$ wih $\theta \in [0,1]$, rather than keeping it fixed we advect it with the Navier-Stokes flow itself. More precisely, given a sufficiently smooth vector field $v$ we define $C_v (\theta, t)$ as 
\begin{equation}\label{e:loop-advected}
\left\{
\begin{array}{ll}
\frac{\partial C_v(\theta, t)}{\partial t} &= v (C_v (\theta, t), t)\\ \\
C_v (\theta, 0) &= C(\theta,0)
\end{array}
\right.
\end{equation}
Hence, assuming that $\mu$-almost surely the trajectory $v$ of Navier-Stokes is smooth enough, we define the liquid loop functional associated to $\mu$ as 
\begin{equation}\label{eq:liquid-loopfunctional}
		\Psi[C,t] := \mathbb{E} \left[ \exp\left\{ \frac{i \gamma}{\nu} \Gamma[u, C_u (\cdot, t) ,t] \right\} \right]
        = \int \left[ \exp\left\{ \frac{i \gamma}{\nu} \Gamma[u, C_u (\cdot, t),t] \right\} \right]\, d\mu (u)\, .
	\end{equation}

\begin{remark}\label{r:no-marginal}
We observe a crucial difference between \eqref{eq:liquid-loopfunctional} and \eqref{eq:loopfunctional}. While the average appearing in \eqref{eq:loopfunctional} depends only on the time-marginal $\mu_t$ of the probability distribution $\mu$, the average in \eqref{eq:liquid-loopfunctional} really depends on the marginal of $\mu$ over the entire time-interval $[0,t]$, because for each different $u$ the loop is evolved according to the flow of $u$ from time $0$ until the specific time $t$.
\end{remark}

The matter of even defining the liquid loop functional in the setting of a probability measure on trajectories which are just suitable weak solutions of Navier-Stokes, is a nontrivial task. On the one hand, the problem of giving a notion of flow for the solution $v$ which goes beyond the (possible) singularity formation has been already considered in the literature, see in particular \cite{RS09,RS09b,Galeati}. The reference \cite{RS09} even shows that for almost every initial point $x$ the particle trajectory of $v$ starting at $x$ avoids the singular set of $v$ for all times, and hence is globally smooth. These results can certainly be used to make sense of $C_v (\theta, t)$ for a typical loop and for a.e. $\theta\in \mathbb S^1$. However, if at time $t_0$ the particle trajectory of a single point of the loop hits a singularity of the flow, the regularity of the loop could be compromised at all future times. While in the formula \eqref{eq:liquid-loopfunctional} we need at least that the loop remains rectifiable for most times. Therefore making sense of \eqref{eq:liquid-loopfunctional} is an interesting mathematical challenge which goes beyond the scope of this note.

\section{Area derivative and loop equation}\label{s:loop-equation}

From this section onward, we will mainly focus on the case $\Omega=\mathbb{R}^3$.
As already mentioned in the introduction, in Migdal's theory knowledge of the loop functional $\Psi[C,t]$ for every $C\in \mathcal{L}\R^3$ allows us to extract information about the underlying probabilistic solution by applying suitable differential operators. We follow Migdal's computations, which assume regular solutions of Navier-Stokes and just notice that, using the representation formula of Proposition \ref{p:loop-regularity}, the regularity almost everywhere of Leray-Hopf weak solutions and a standard Fubini-Tonelli argument, it is possible to extend the computations to the low-regularity setting of the previous sections, at least when the probabilistic solution is translation-invariant. 

\subsection{Area Derivative and Vorticity} First of all, we extend the functional $\Gamma [u, C, t]$ to finite formal sums of loops $\sum_i C_i$ simply by setting 
\[
\Gamma \left[u, \sum_i C_i, t\right] = \sum_i \Gamma [u, C_i, t]\, .
\]
We next follow Migdal and introduce the area derivative at a point $x$ for a smooth vector field $u$ (albeit following a somewhat different formalism, the reader is encouraged to compare our computation with \cite{Migdal23}; see also Remark \ref{r:comparison} below). We fix a oriented plane $\pi$ with unit normal vector $n$, a circle $\gamma_r$ of radius $r$ centered at $x$ in the plane $\pi$ parametrized in the counterclowise direction and compute the limit
\begin{align*}
&\lim_{r\downarrow 0} \frac{1}{\pi r^2}\left(
\exp \left\lbrace \frac{i \gamma}{\nu} \Gamma[u, C+ \gamma_r,t]\right\rbrace 
-
\exp \left\lbrace \frac{i \gamma}{\nu} \Gamma[u, C, t])\right\rbrace \right)\\
= &
\exp \left\lbrace \frac{i \gamma}{\nu} \Gamma[u, C, t])\right\rbrace
\lim_{r\downarrow 0} \frac{1}{\pi r^2} 
\left(
\exp \left\lbrace \frac{i \gamma}{\nu} \Gamma[u, \gamma_r,t]\right\rbrace - 1
\right)\, .
\end{align*}
Next, if we consider the disk $D_r$ bounded by the circle $\gamma_r$, we can use Stokes' Theorem to compute
\[
\Gamma [u, \gamma_r, t] = \int_{D_r} \omega (y,t)\cdot n\, dy 
\]
and in particular we infer that 
\[
\lim_{r\downarrow 0} \frac{1}{\pi r^2} 
\left(
\exp \left\lbrace \frac{i \gamma}{\nu} \Gamma[u, \gamma_r,t]\right\rbrace - 1
\right) 
= \omega (x) \cdot n\, .
\]
Following Migdal, if we consider an oriented coordinate plane $\pi$ spanned by the orthonormal base $e_\alpha, e_\beta$ (consisting of basis vectors), we encode the above computations in the formula
\[
\frac{\delta}{\delta \sigma_{\alpha \beta} (x)} 
\exp \left\lbrace \frac{i \gamma}{\nu} \Gamma[u, C, t])\right\rbrace 
= \bar{\omega}_{\alpha \beta} (x) \exp \left\lbrace \frac{i \gamma}{\nu} \Gamma[u, C, t])\right\rbrace 
\]
where 
\begin{equation}\label{e:omega}
\bar \omega_{\alpha \beta} = \partial_\alpha v_\beta - \partial_\beta v_\alpha\, .
\end{equation}
Clearly the antisymmetric tensor $\bar\omega$ is related to the vorticity $\omega$ (a $3$-dimensional vector field) through the usual formula
\[
\bar\omega_{\alpha \beta} = \varepsilon_{\alpha \beta i} \omega_i\, ,
\]
where, as usual $\varepsilon_{\alpha \beta i}$ is the totally antisymmetric tensor (which vanishes unless the triple $(\alpha, \beta, i)$ is a permutation of $(1,2,3)$ and takes the signature of the permutation otherwise).  

\begin{remark}\label{r:comparison}
We warn however the reader that in fact Migdal's formalism is somewhat different from the above (see, for instance, \cite[Page~5]{Migdal23}): the perturbation of the loop $C$ is performed using the \emph{spike operator}, rather than by summing small closed loops as we did above.
\end{remark}

\subsection{Navier-Stokes in Vorticity Formulation} We next use the following well-known reformulation of the Navier-Stokes equations:

\begin{lemma}\label{l:lemma-vorticity}
Assume $(v,p)$ is a Leray-Hopf weak solution of Navier-Stokes and let $\bar \omega$ be the antisymmetric tensor given by \eqref{e:omega}. Then
\[
\partial_t v_\alpha = \nu \partial_\beta \bar \omega_{\beta \alpha} - v_\beta \bar \omega_{\beta \alpha} - \partial_\alpha \left(\frac{|v|^2}{2}+p\right)\, , 
\]
or equivalently
\begin{equation}\label{eq:NS2}
    \partial_t v = \nu \nabla \times \omega + \omega \times v - \nabla \left( \frac{1}{2}|v|^2 + p \right)\, ,
\end{equation}
where the identity can be interpreted in the sense that each summand appearing in the equation is in fact a function in $L^{5/4}$. Moreover 
\[
v_\alpha (x) = \Delta^{-1} \partial_\beta \bar \omega_{\beta\alpha} (x) :=
\left(\frac{1}{4\pi |\cdot|} * \partial_\beta \bar \omega_\beta \alpha\right) (x)\, 
\]
or alternatively
\[
v (x) = \left(\frac{1}{4\pi} |\cdot| * \nabla \times \omega \right) (x) = \left(\nabla \frac{1}{4\pi |\cdot|} \times \omega \right) (x)
\]
\end{lemma}

\begin{proof} This is a very well known fact, but we include a proof for the reader's convenience. First of all 
recall that $v\in L^{\frac{10}{3}}$ and $Dv \in L^2$. In particular $|\nabla |v|^2|\leq |v||Dv|$ and $|v||Dv|\in L^{5/4}$. Using that $p$ is the potential-theoretic solution of $- \Delta p = \partial_\alpha (v_\beta \partial_\beta v_\alpha)$, the Calderon-Zygmund estimates imply also $\nabla p \in L^{\frac{5}{4}}$. We next use 
\[
(\partial_t -\Delta) v = - \nabla p - (v\cdot \nabla) v
\]
and classical estimates on the heat equation to infer that $\partial_t v$ and $D^2 v$ both belong to $L^{\frac{5}{4}}$. Since $|D\bar \omega|\leq |D^2 v|$, the claim that all summands in the identity are $L^{5/4}$ functions is now obvious. As for the identity itself, it follows from simple calculus rules on every region where the pair $(v,p)$ is smooth. However, by Leray's theorem on epochs of regularity, such smoothness is guaranteed on an open set whose complement has measure zero. Thus the first claim of the lemma follows easily.

Finally, the second relation is the Biot-Savart law to recover $v$ knowing $\bar \omega$, which is a simple consequence of the relation $\Delta v_\alpha = \partial_\beta \bar \omega_{\beta\alpha}$ (it suffices to prove it on the regular set of the solution) and the summability of $v$, which ensures that the harmonic function $v_\alpha - \Delta^{-1} \partial_\beta \bar \omega_{\beta\alpha}$ must vanish identically.
\end{proof}

\subsection{Loop Equation} Assume now that $(u,p)$ is a Leray-Hopf weak solution of the Navier-Stokes equations and that $t$ is a time with the property that the pair $(u,p)$ is smooth on $\mathbb R^3 \times (t-\varepsilon, t+\varepsilon)$. We can then compute 
\begin{align}
& \frac{d}{dt} \exp \left\lbrace \frac{i \gamma}{\nu} \Gamma [u, C, t]\right\rbrace \nonumber\\
= & \exp \left\lbrace \frac{i \gamma}{\nu} \Gamma [u, C, t]\right\rbrace \frac{i\gamma}{\nu}\int_0^1 \partial_t u (C(\theta), t)\cdot C' (\theta)\, d\theta 
\nonumber\\
= &  \exp \left\lbrace \frac{i \gamma}{\nu} \Gamma [u, C, t]\right\rbrace \frac{i \gamma}{\nu} \int_0^1   C_\alpha' (\theta)\cdot \big[(\nu \partial_\beta \bar \omega_{\alpha\beta} \omega - u_\beta \partial_\beta \bar \omega_{\alpha \beta}) (C(\theta), t)\big]\, d\theta\, \label{e:loop-derivative}
\end{align}
where we have used that 
\begin{align*}
\int_0^1  C'_\alpha (\theta) \partial_\alpha \left(\frac{|u|^2}{2} + p \right) (C(\theta), t)d\theta
= & \int_0^1 \frac{d}{d\theta} \left(\left(\frac{|u|^2}{2} + p\right) (C (\theta), t)\right)\, d\theta = 0\, .
\end{align*}
We can therefore further write 
\begin{align}
& \frac{d}{dt} \exp \left\lbrace \frac{i \gamma}{\nu} \Gamma [u, C, t]\right\rbrace 
\nonumber\\
= & \frac{i\gamma}{\nu}\int_0^1 \Big[\nu \frac{\partial}{\partial x_\beta} \frac{\delta}{\delta \sigma_{\beta \alpha}(x)} 
\exp \left\lbrace \frac{i \gamma}{\nu} \Gamma [u, C, t]\right\rbrace\nonumber\\
&\qquad
- \frac{\delta}{\delta \sigma_{\beta\alpha}(x)} \Big( \Delta^{-1}_x\frac{\partial}{\partial x_\gamma} \frac{\delta}{\delta \sigma_{\gamma \beta}(x)} 
\exp \left\lbrace \frac{i \gamma}{\nu} \Gamma [u, C, t]\right\rbrace
\Big) \Big] \Big|_{x=C(\theta)} C'_\alpha (\theta)\, d\theta\, .\label{e:pointwise}
\end{align}
We can now use Proposition \ref{p:loop-regularity} to write a similar identity for a probabilistic admissible weak solution which is invariant under translations.

\begin{corollary}\label{c:loop-equation}
	Let $\mu$ be a probabilistic suitable weak solution of the Navier–Stokes equations \eqref{NS} in $\mathbb{R}^3$, with an initial distribution $\mu_0$ such that 
    \begin{equation}
    \int \|u\|_{L^2}^2\, d\mu_0 (u) < \infty\, .
    \end{equation}
    If in addition $\mu$ is invariant under translations we then have the identity
    \begin{align}
    \frac{d}{dt} \Psi [C, t]
    = & \int_0^1 \int \phi (y) \Big[\frac{\partial}{\partial x_\beta} \frac{\delta}{\delta \sigma_{\beta \alpha}(x)} \Psi[ C+y, t] \nonumber\\
    &\qquad- \frac{\delta}{\delta \sigma_{\beta\alpha}(x)} \Big( \Delta^{-1}_x\frac{\partial}{\partial x_\gamma} \frac{\delta}{\delta \sigma_{\gamma \beta}(x)} \Psi [C+y, t]\Big)\Big]\Big|_{x=C(\theta)}  \, dy\, C'_\alpha (\theta)\, d\theta\, .\label{e:weak-loop-equation}
    \end{align}
\end{corollary}
\begin{proof}
We consider the measure $\mu \otimes dt$ and observe that, since $\mu$ is concentrated on Leray-Hopf weak solutions, the set $\mathcal{S}$ of pairs $(u,t)$ for which $t$ is a singular time of the solution $u$, is a null set. In particular, we can use the validity of \eqref{e:pointwise} outside of the set $\mathcal{S}$, Fubini-Tonelli with the representation formula \eqref{e:representation-formula} and the linearity in $u$ of the operators appearing in the identity to complete the proof.
\end{proof}

\begin{remark}\label{r:marginali?}
Note that the Corollary \ref{c:loop-equation} claims only ``one direction'': if we have a measure which is concentrated on a set of suitable weak trajectories of Navier-Stokes, then the loop functional computed on every single specific loop solves \eqref{e:weak-loop-equation}. An interesting and highly nontrivial question is whether an evolving stochastic process $\mu_t$ with the property that the corresponding loop functionals solve \eqref{e:weak-loop-equation} for every loop must be necessarily the time marginal of a measure concentrated on trajectories of Navier-Stokes. This would be implied a-posteriori by very strong uniqueness theorems: {\em if} the solutions to the Cauchy problem of Navier-Stokes were unique and {\em if} the equation \eqref{e:weak-loop-equation} (satisfied for every loop) were to uniquely determine the evolving $\mu_t$ from $\mu_0$, then necessarily $\mu_t$ would have to be the time-marginal of the push-forward of the initial distribution $\mu_0$ through the Navier-Stokes semigroup.    
\end{remark}

\subsection{Liquid Loop Equation}\label{s:liquid-equation}
We report here the elementary computations pertaining to the case in which the loop is advected by the flow. Consider in particular a smooth solution $v$ of the Navier-Stokes equation and an evolving loop $C_v$ satisfying \eqref{e:loop-advected}. In this case, we have the identity 
\begin{align*}
\frac{d}{dt} \Gamma [v, C_v (\cdot, t), t] & = \nu \int_0^1 (\nabla \times \omega) (C_v (\theta, t), t)\cdot \frac{\partial C_v}{\partial \theta} (\theta, t)\, d\theta\, ,
\end{align*}
which is just Kelvin's circulation theorem (cf. \cite[Page 23]{MajdaBertozzi}). We can therefore use the area derivative to write 
\begin{align*}
 & \frac{d}{dt} \exp \left\lbrace \frac{i \gamma}{\nu} \Gamma [u, C_u (\cdot, t), t]\right\rbrace \\
= & i \gamma \int_0^1 \Big[\frac{\partial}{\partial x_\beta} \frac{\delta}{\delta \sigma_{\beta \alpha}(x)} \exp \left\lbrace \frac{i \gamma}{\nu} \Gamma [u, C, t]\right\rbrace\Big]\Big|_{C = C_u (\cdot, t),\, x= C_u (\theta, t)} \frac{\partial C^\alpha_u}{\partial \theta} (\theta,t)\, d \theta\, .
\end{align*}
Note that in the right hand side we understand that the area derivative at the position $x$ is taken on the functional 
\[
\exp \left\lbrace \frac{i \gamma}{\nu} \Gamma [u, C , t] \right\rbrace
\]
in the loop variable $C$. {\em Aftewards} it is evaluated at the evolved loop $C=C_u (\cdot, t)$ and at the point $x= C_u (\theta, t)$.

\section{Momentum Variables}

Migdal, in order to find interesting solutions of the loop equation, introduces an ``infinite-dimensional Fourier transform'' of the probabilistic solutions $\mu_t$. A way to formalize the latter is to hope for the existence of a measure $\beta$ on the space of (weakly) continuous functions $C_w ([0,T]; X)$ taking values in some space $X$ of generalized functions $P: \mathbb S^1 \to \mathbb C^3$ with the property that 
\begin{equation}\label{eq:loopmomentum}
	\Psi[C,t] = \int \exp\left\lbrace \frac{i \gamma}{\nu} \int_0^1 P(\theta, t) \cdot C'(\theta)\, d\theta \right\rbrace\, d\beta_t (P)
	\quad \text{for every $C\in \mathcal{L}\R^n$}\, 
\end{equation}
where $\beta_t = (e_t)_\sharp \beta$ (following the notation introduced in Section \ref{s:probabilistic-solutions}).
We stress that we are not restricting our search on $P$'s which are classical functions of the variable $\theta$: we rather interpret the integral 
\[
\int_0^1 P (\theta, t)\cdot C' (\theta)\, d\theta
\]
as the action of the distribution $P (\cdot, t)$ on the test function $C' (\theta)$. Note that we can allow $P$ to have rather negative order if we impose that $C$ is regular enough, in particular we could search for probability measures $\beta_t$ in Banach spaces of negative distributions such as $H^{-k}$ for $k$ large (where, as usual, we denote by $H^{-k}$ the dual of the Hilbert space $H^k$). 

In this section we show that just some minimal requirements on the random variable $\beta_t$ obstruct its existence. 
The key is that it is not possible to have \eqref{eq:loopmomentum} and a constant imaginary part for $P$, as Migdal is assuming in, e.g., \cite{Migdal24}.

\begin{theorem}\label{t:non-existence}
Consider a probability measure $\mu$ on the space of $C^2$ vector fields with the property that 
\[
\int \|u\|_{C^2} \, d\mu (u) < \infty
\]
and a probability measure $\beta$ on the space $H^{-k} (\mathbb S^1, \mathbb C^3)$ with the property that 
\[
\int \|P\|_{H^{-k}}\, d\beta (P) < \infty\, .
\]
Assume that, for every $C^\infty$ embedded loop $C$ in $\mathbb R^3$ parametrized over $[0,1]$ we have the identity
\[
\int \exp \Big( i \int_0^1 v (C(\theta))\cdot C' (\theta)\, d\theta\Big)\, d\mu (v) 
= \int \exp \Big (i \int_0^1 P (\theta) \cdot C' (\theta)\, d\theta \Big)\, d\beta (P)
\]
and that the imaginary part of $P$ is constant in $\theta$ for $\beta$-almost all $P$. Then $\mu$ must be concentrated on curl-free vector fields or, in other words, the circulation
\[
\int_0^1 v (C(\theta)) \cdot C' (\theta)\, d\theta
\]
vanishes for every loop $C$ and $\mu$-almost all $v$.
\end{theorem}

\begin{proof}
We introduce the notation
\begin{align}
\psi (v,C) &:= \exp \left(i \int_0^{1} v (C (\theta))\cdot C' (\theta)\, d\theta\right)\, ,\label{e:loop}\\
\phi (P,C) &:= \exp \left(i \int_0^{1} P (\theta) \cdot C' (\theta)\, d\theta\right)\, .
\end{align}
Consider now an arbitrary point $x$ in $\mathbb R^2$ and an arbitrary smooth embedded loop $\gamma: \mathbb S^1 \to \mathbb R^3$. We then define the functions
\begin{align*}
G (x, \sigma, v) &:= \psi (v, x+\sigma \gamma)\\
F (x, \sigma, P) &:= \phi (P, x+\sigma \gamma)\, ,
\end{align*}
where $\sigma >0$ is an arbitrary positive real number and $x+\sigma \gamma$ denotes the loop
$\theta \mapsto x+\sigma \gamma (\theta)$.
We first Taylor expand the function $F(x, \sigma, P)$ in $\sigma$ and we find
\begin{equation}\label{e:Taylor-1}
F (x, \sigma, P) = 1 + i \sigma \int_0^1 P (\theta)\cdot \gamma' (\theta)\, d\theta
- \frac{\sigma^2}{2} \Big(\int_0^1 P (\theta)\cdot \gamma' (\theta)\, d\theta\Big)^2 + O (\sigma^3)\, ,
\end{equation}
where the error $O (\sigma^3)$ can be estimated by 
\begin{equation}\label{e:Taylor-2}
|O (\sigma^3)|\leq \sigma^3 (\|P\|_{H^{-k}} \|\gamma\|_{H^{k+1}})^3\, . 
\end{equation}
We next write 
\begin{align*}
& \int_0^1 v (x+\sigma \gamma (\theta)) \cdot (\sigma \gamma' (\theta))\, d\theta\\
= & \sigma v (x) \cdot \underbrace{\int_0^1 \gamma' (\theta)\, d\theta}_{=0} 
+ \sigma^2 \int_0^1 (\gamma(\theta)\cdot \nabla) v (x) \cdot \gamma' (\theta)\, d\theta 
+ O (\sigma^3)\, \\
= & \omega (x) \cdot \underbrace{\int_0^1 \gamma' (\theta) \times \gamma (\theta)\, d\theta}_{:= A (\gamma)} + O (\sigma^3)
\end{align*}
where 
\begin{equation}\label{e:Taylor-4}
|O (\sigma^3)|\leq \sigma^3 \|v\|_{C^2} \|\gamma\|_{C^1}\, .
\end{equation}
We can therefore Taylor expand $G$ in $\sigma$ as 
\begin{equation}\label{e:Taylor-3}
G (x, \sigma, v) = i \sigma^2 \omega (x) \cdot A (\gamma) + O (\sigma^3)
\end{equation}
where in fact the error $O (\sigma^3)$ still obeys \eqref{e:Taylor-4}. We next integrate \eqref{e:Taylor-3} in $v$ with respect to the measure $\mu$ and \eqref{e:Taylor-1} in $\beta$. Using the assumptions of the theorem we can equate the terms in $\sigma^2$ in the corresponding expansions leading to:
\begin{equation}\label{e:equate-Taylor}
2 i \int A (\gamma) \cdot \omega (x)\, d\mu (u) = - \int \Big(\int_0^1 P (\theta)\cdot \gamma' (\theta)\, d\theta\Big)^2\, d\beta (P)\, .
\end{equation}
Now observe that, if the imaginary part of $P$ is constant almost surely then 
\[
{\rm Im}\, \Big(\int_0^1 P (\theta) \cdot \gamma' (\theta)\, d\theta\Big) = 0
\qquad \mbox{for $\beta$-a.e. $P$.}
\]
So, the left hand side of \eqref{e:equate-Taylor} is real, while the left hand side is purely imaginary, unless $\omega (x) \cdot A(\gamma) =0$ for $\mu$-a.e. $u$. Since we can vary $x$ on a countable dense subset of $\mathbb R^3$ we can assume that the identity holds for a fixed $\gamma$, a subset of points $x$ dense in $\mathbb R^3$ and all but a $\mu$-null subset of solutions $u$. 
Note that, if we choose $\gamma$ to be an embedded loop contained in a plane $\pi$, then $A(\gamma)$ is a vector orthogonal to $\pi$ whose modulus is the area of the bounded open region of $\pi$ with boundary $\gamma$. We conclude therefore that $\mu$-almost surely the curl of $u$ has to vanish on a dense subset of $\mathbb R^3$. But since $u$ is, $\mu$-almost surely, $C^2$, we conclude that $u$ is, $\mu$-almost surely, irrotational. This completes the proof. 
\end{proof}

\section{Discrete Framework}\label{s:Fourier}

Fix $N \in \mathbb{N}$, $N\ge 3$. Let $C_0, \ldots, C_{N-1} \in \mathbb{R}^3$ be points interpreted as the vertices of an $N$-gon. We define the sides by $\triangle C_k := C_{k+1} - C_k$, using the cyclic conventions $C_0 = C_N$ and $C_{-1} = C_{N-1}$. 
With a slight abuse of notation, 
We denote by $C$ the piecewise linear loop obtained by traversing consecutively the segments 
$[C_0, C_1]$, $[C_1, C_2]$,..., $[C_{N-1}, C_0]$,
so that $C$ is a closed, oriented polygonal curve with vertices $C_0, \ldots, C_{N-1}$.

For parameters $0 \le p < q \le N-1$, we define
\begin{equation}\label{eq:L}
    L_{p,q} := \sum_{k=p}^q |\triangle C_k|.
\end{equation}
Note that $L_{0,N-1}$
corresponds to the total length $|C|$ of the polygonal loop $C$.

\subsection{Discretized Momentum Variables}\label{subsec:discretizedmomentum}

We let $\mathcal{L}_N \R^3$ denote the subset of $\mathcal{L} \R^3$ consisting of piecewise linear loops. 
Note that $\mathcal{L}_N \R^3$ is finite-dimensional, and the restriction of the loop functional $\Psi[C,t]$ (see Definition~\ref{def:loopfunctional}) to $\mathcal{L}_N \R^3$ can be identified with a function of $N$ variables, determined by the vertices $C_0, \ldots, C_{N-1}$.

When we restrict attention to loops in $\mathcal{L}_N \mathbb{R}^3$, we can also consider a discretization of the momentum variables. Given a loop functional $\Psi[C,t]$, we say that an $N$-dimensional complex stochastic process $(P_0(t), \ldots, P_{N-1}(t))$ taking values in $\C^N$ with law $\beta_t$
is a set of \emph{momentum variables} if
\begin{equation}\label{eq:loopmomentumdiscrete}
	\Psi[C,t] = \int \left[ \exp\left\{ \frac{i \gamma}{\nu} \sum_{k=0}^{N-1} P_k \cdot \triangle C_k \right\} \right]
    \, d\beta_t (P)
\end{equation}
for every $C \in \mathcal{L}_N \mathbb{R}^3$. The obstruction of the previous section still applies and in particular we cannot hope to find a stochastic process with (almost surely) constant imaginary part and the property that 
\begin{equation}\label{e:Fourier-transform}
\Psi [C,t] = \int \exp \left\lbrace \frac{i \gamma}{\nu} \int_0^1 u (C(\theta), t)\cdot C' (\theta)\, d\theta \right\rbrace\, d\mu_t (u)\, 
\end{equation}
for all $C\in \mathcal{L}_N \mathbb R^3$.

Observe, however, that the right-hand side of \eqref{e:Fourier-transform} is a function of the finitely many variables $C_0, \ldots, C_{N-1}$. But in fact, under the assumption that $\mu_t$ is translation invariant, the 
right-hand side of \eqref{e:Fourier-transform} is a function $F$ of the variables $\triangle C_0, \ldots, \triangle C_{N-1}$, where $\triangle C_1, \ldots, \triangle C_{N-1}$ is a set of independent variables owing to the fact that 
\[
\triangle C_0 = - \sum_{i=1}^{N-1} \triangle C_i\, ,
\]
as it can be concluded from the relation $C_N = C_0$. We therefore write 
\[
\int \exp \left\lbrace \frac{i \gamma}{\nu} \int_0^1 u (C(\theta), t)\cdot C' (\theta)\, d\theta \right\rbrace\, d\mu_t (u) = G (\triangle C_1, \ldots, \triangle C_{N-1}, t)\, , 
\]
for some function $G$ on $(\mathbb R^3)^{N-1} \times \mathbb R$. 
Observe next that the expression
\[
\exp\left\{ \frac{i \gamma}{\nu} \sum_{k=0}^{N-1} P_k \cdot \triangle C_k \right\}
\]
can be written as 
\[
\exp \left\{ \frac{i\gamma}{\nu} \sum_{k=1}^{N-1} (P_k - P_0) \cdot \triangle C_k\right\}\, . 
\]
Given the translation invariance of the measure we are interested in, we can, and will, impose $P_0 =0$. Rather than looking for a probability measure $\beta_t$ on $(\mathbb R^3)^{N-1}$, we can now look for a {\em signed measure} (or more generally some generalized function) satisfying the relation
\begin{equation}\label{e:vero-beta}
 \int \exp \left\lbrace \frac{i\gamma}{\nu} \sum_{k=1}^{N-1} P_k \cdot \triangle C_k \right\rbrace\, d\beta_t (P)
 = G (\triangle C_1, \ldots , \triangle C_{N-1} , t)\, . 
\end{equation}
In particular $\beta_t$ is (up to the constant factor $\frac{\gamma}{\nu}$) just the Fourier transform of $G (\cdot, t)$.

\section{Discretized loop calculus}

In the rest of our work we will analyze the counterpart of the considerations of Section \ref{s:loop-equation} in the discretized setting of the previous section. This framework has been suggested by Migdal in \cite{Migdal24b}. The aim is to derive a suitable approximate form of the loop equation of Section \ref{s:loop-equation}. Such a discretized loop equation will contain error terms which we will show are negligible in (powers of) $N$, under the assumption that the process $\mu_t$ is concentrated on sufficiently regular vector fields. Eventually we will take the Fourier transform of Section \ref{s:Fourier} and give some validation to Migdal's discretized Euler ensemble, which is a set of particular solutions to the discretized loop equation.

\medskip

Fix a large natural number $N$ and a loop $C \in \mathcal{L}_N \R^3$. 
Recall that $C$ is determined by $N$ vertices, denoted $C_0, \ldots, C_{N-1}$, and adopt the \emph{cyclic convention} $C_{k+N} = C_k$. 
In this way, we may use integer labels $k \in \mathbb Z$ without restricting to the range $\{0,\ldots,N-1\}$.

Given a velocity field $u(x)$ and $C \in \mathcal{L}_N \R^3$, we consider the circulation of $u$ along $C$, viewed as a function of the vertices $C_0, \ldots, C_{N-1}$:
\begin{equation}
	\Gamma[u,C] \;=\; \Gamma[u, C_0, \ldots, C_{N-1}] 
	= \int_C u \cdot dx.
\end{equation}
In this section the velocity field $u$ is not assumed to have any link with the Navier-Stokes equation (it is not even assumed to be divergence-free). We rather assume that $u$ is regular enough, namely of class $C^K$ for a $K$ large enough to cover that all the derivatives appearing in the statements are continuous functions. 

\subsection{Partial Derivatives} 
We are interested in computing the partial derivatives of $\Gamma$  in the variables $C_0, \ldots, C_{N-1}$. We begin by observing that, in order to compute $\nabla_{C_k} \Gamma$, given the additivity of the path integral, it is enough to study the derivative of the circulation around the boundary of the triangle $T_k$ with vertices $C_{k-1}, C_k, C_{k+1}$. To this end, by Stokes theorem we write
\begin{equation}
	\Gamma[u, C_{k-1}, C_k, C_{k+1}]
	= \int_{T_k} \omega \cdot n\, d\sigma
\end{equation}
where $\omega = \nabla \times u$ is the vorticity of $u$, and 
\begin{equation}
    n = \frac{(C_{k+1} - C_k) \times (C_{k-1} - C_k)}{|(C_{k+1} - C_k) \times (C_{k-1} - C_k)|}
    = - \frac{\triangle C_k \times \triangle C_{k-1}}{|\triangle C_k \times \triangle C_{k-1}|}
\end{equation}
is the normal to $T_k$.
It is convenient to parametrize $T_k$ with coordinates 
\begin{equation}
	f(a,b) = a C_k + (1-a) \frac{C_{k+1}+C_{k-1}}{2} + b (C_{k+1}-C_{k-1})
\end{equation}
where for each $a\in [0,1]$ the parameter $b$ ranges between $-\frac{1-a}{2}$ and $\frac{1-a}{2}$. Observe that the area element in this parametrization is precisely the modulus of the vector
\begin{align*}
\left(C_k-\frac{C_{k+1}+C_{k-1}}{2}\right) \times (C_{k+1} - C_{k-1})
& = \frac{1}{2} (\triangle C_{k-1} - \triangle C_k) \times (\triangle C_k+\triangle C_{k-1})\\
&= \triangle C_{k-1} \times \triangle C_k\, .
\end{align*}
We can therefore write 
\begin{equation}
	\Gamma[u, C_{k-1}, C_k, C_{k+1}]
	=
	-\int_0^1 \left(\int_{-\frac{1-a}{2}}^{\frac{1-a}{2}}\, \omega(f(a,b))\cdot ( \triangle C_k \times \triangle C_{k-1}) db\right)da,
\end{equation}
concluding that
\begin{align}
	\nabla_{C_k} \Gamma
	& =
	- \int_0^1\left( \int_{-\frac{1-a}{2}}^{\frac{1-a}{2}} \omega(f(a,b))  \times (C_{k+1} - C_{k-1}) db\right) da\nonumber
	\\&-
	\int_0^1 a \left( \int_{-\frac{1-a}{2}}^{\frac{1-a}{2}} D\omega(f(a,b))^T \cdot ( \triangle C_k \times \triangle C_{k-1}) db\right)da\label{e:formulozza},
\end{align}
for every $k=1, \ldots, N-1$.

\begin{lemma}\label{l:nablaGamma_bound}
    With the above notation, for $0\leq k \leq N-1$, $K\geq 0$ integer and $u\in C^{K+2}$ we have:
    \begin{equation}\label{eq:nablaGamma_bound}
    \begin{split}
        &|D^K_{C_k} \nabla_{C_k}\Gamma[u, C_0, \ldots, C_{N-1}]|
        \le \const(K) |C_{k+1} - C_{k-1}|(1 + |C_k - C_{k-1}|) \| \omega\|_{C^{K+1}} 
        \\
        &|D_{C_{k+1}}\nabla_{C_k}\Gamma[u, C_0, \ldots, C_{N-1}]| 
        \le \const (1 + |C_{k+1} - C_{k-1}|)(1 + |C_k - C_{k-1}|) \| \omega \|_{C^2}
        \\
        & |D_{C_{k-1}}\nabla_{C_k}\Gamma[u, C_0, \ldots, C_{N-1}]| 
        \le 
        \const (1 + |C_{k+1} - C_{k-1}|)(1 + |C_{k+1} - C_{k}|) \| \omega \|_{C^2}
    \end{split}
    \end{equation} 
\end{lemma}

\begin{proof}
The first integrand in \eqref{e:formulozza} clearly satisfies \eqref{eq:nablaGamma_bound}. 
To control the second integrand, we observe that
\begin{equation}
    \triangle C_k \times \triangle C_{k-1} 
    = (C_{k+1}-C_{k-1}) \times \triangle C_{k-1}.
\end{equation}

\end{proof}

\subsection{Discretized Area Derivative}
The next goal is to compute 
\begin{equation}
	\nabla_{C_{k+1}} \times \nabla_{C_k} \Gamma\, ,
\end{equation}
which plays the role of a discrete Area derivative. The lemma below follows a suggestion of (and uses some computations by) Migdal in a private communication to the authors, \cite{Migdal25}. We warn the reader that Migdal takes a different definition of discretized area derivative in his works, cf. \cite{Migdal24b} and Section \ref{s:Sasha}. The leading terms in both definitions coincide, but the errors terms seem to behave in a different way when the area derivative is composed with other operators, cf. Section \ref{s:advection}.  

\begin{lemma}\label{lemma:nabla x nabla 1}
	With the above notation, for $0\leq k \leq N-1$ and any integer $K\geq 0$, if $u\in C^{K+2}$ we then have:
	\begin{equation}\label{e:finale}
		\nabla_{C_{k+1}} \times \nabla_{C_k} \Gamma 
		= - \int_0^1 \omega (aC_k + (1-a) {\textstyle{\frac{C_{k-1} + C_{k+1}}{2}}})\, da + R^k_{\rm ad}\, ,
	\end{equation}
    where $R^k_{\rm ad}$ (the area derivative error) satisfies the estimate 
    \begin{equation}\label{est:R_ad}
    |D^K_{C_k} R^k_{\rm ad}| (C_0, \ldots , C_{N-1}) \leq \const(K) \|D^{K+1} \omega\|_{L^\infty} |C_{k+1} - C_{k-1}|\, .
    \end{equation}
\end{lemma}

\begin{proof}
We start by observing that 
\[
\nabla_{C_{k+1}} \times \nabla_{C_k} \Gamma 
= (\nabla_{C_{k+1}} + {\textstyle{\frac{1}{2}}} \nabla_{C_k})\times \nabla_{C_k} \Gamma\, .
\]
We now apply the operator $\nabla_{C_{k+1}} + \frac{1}{2} \nabla_{C_k}$ to the fist integrand 
\[
G = \omega (f(a,b)) \times (C_{k+1} - C_{k-1})
\]
which yields
\[
(\nabla_{C_{k+1}} + {\textstyle{\frac{1}{2}}} \nabla_{C_k})\times G 
= 2 \omega (f(a,b)) + R^k_1 
\]
where
\begin{equation}
\begin{split}
   R^k_1 &:= (C_{k+1} - C_{k-1}) \cdot \left( \nabla_{C_{k+1}} + \frac{1}{2} \nabla_{C_k} \right) \omega (f(a,b))
    \\&
    = (b + \frac{1}{2} a) ((C_{k+1} - C_{k-1}) \cdot \nabla \omega)(f(a,b))\, .  
\end{split}
\end{equation}
The letter clearly is a linear combinations of partial derivatives of $\omega$ (computed at some specific points) multiplied by components of the vector $(C_{k+1}-C_{k-1})$. We therefore gain the estimate 
\[
|D^K_{C_k} R^k_1| \leq \const (K) \|D^{K+1} \omega\|_{L^\infty} |C_{k+1}-C_{k-1}|
\]

\medskip

We now apply the operator to the second integrand in \eqref{e:formulozza}. Up to factors of $a$ we write it, using the antisymmetric tensor $\varepsilon_{\alpha \beta \gamma}$, as 
\[
F := \partial_j \omega^i (f(a,b)) \varepsilon_{i\alpha \beta} (C_{k+1}-C_k)^\alpha (C_k-C_{k-1})^\beta
\]
We next compute the curls in $C_{k+1}$ and $C_k$. For the curl in $C_{k+1}$ we get 
\begin{align*}
(\nabla_{C_{k+1}} \times F)^j = &\varepsilon_{j \gamma \ell} \partial^{C_{k+1}}_\gamma \partial_\ell \left[\omega^i (f(a,b)) \varepsilon_{i\alpha \beta} (C_{k+1}-C_k)^\alpha (C_k-C_{k-1})^\beta\right]\\
= & \left(\frac{1-a}{2} +b\right) \varepsilon_{j\gamma \ell} \partial^2_{\gamma \ell} \omega^i (f(a,b)) \varepsilon_{i\alpha\beta} (C_{k+1}-C_k)^\alpha (C_k-C_{k-1})^\beta\\
& + \varepsilon_{j\alpha \ell} \partial_\ell \omega^i (f(a,b)) \varepsilon_{i\alpha \beta} (C_k-C_{k-1})^\beta\\
=&  \varepsilon_{j\alpha \ell}  \varepsilon_{i\alpha \beta} \partial_\ell \omega^i (f(a,b)) (C_k-C_{k-1})^\beta\, .
\end{align*}
We next use the identity
\[
\varepsilon_{j\alpha \ell} \varepsilon_{i \alpha \beta} = \delta_{ji} \delta_{\ell \beta}
- \delta_{j\beta} \delta_{i \ell}
\]
to get 
\begin{align*}
\nabla_{C_k+1} \times F &= ((C_k-C_{k-1}) \cdot \nabla) \omega) (f(a,b)) - (C_k - C_{k-1}) ({\rm div}\, \omega) (f(a,b))\\
&=  ((C_k-C_{k-1}) \cdot \nabla) \omega) (f(a,b))\, .
\end{align*}
On the other hand for computing $\nabla_{C_k} \times F$ we can use the same strategy and just notice, given the antisymmetry of the vector product 
\[
\nabla_{C_k} \times F 
= (C_k-C_{k-1}) - (C_{k+1}-C_k) \cdot \nabla) \omega (f(a,b)).
\]
In particular, summing the two terms we easily conclude that 
\[
(\nabla_{C_{k+1}} + {\textstyle{\frac{1}{2}}} \nabla_{C_k})\times F 
= ({\textstyle{\frac{1}{2}}} (C_{k+1}+C_{k-1} - 2C_k)\cdot \nabla) \omega) (f(a,b))\, .
\]
Collecting all the estimates above, we obtain the identity
\begin{align*}
&(\nabla_{C_{k+1}} + {\textstyle{\frac{1}{2}}} \nabla_{C_k}) \times \nabla_{C_k} \Gamma \\
=& - 2 \int_0^1 \int_{-\frac{1-a}{2}}^{\frac{1-a}{2}} \omega (a C_k + (1-a) {\textstyle{\frac{C_{k-1} + C_{k+1}}{2}}} + b(C_{k+1} - C_{k-1}))\, dadb \\
& -\int_0^1 a \int_{-\frac{1-a}{2}}^{\frac{1-a}{2}}({\textstyle{\frac{C_{k+1}+C_{k-1}-2C_k}{2}}}\cdot \nabla \omega) (C_k + (1-a) {\textstyle{\frac{C_{k+1}+C_{k-1}-2C_k}{2}}} + b(C_{k+1} - C_{k-1}))\, da db\\
& + R^k_1 
\end{align*}
where $R^k_1$ satisfies \eqref{est:R_ad}.
We next approximate the integral in $b$ with $(1-a)$ times the integrand in the point 
\[
a C_k + (1-a) \frac{C_{k+1}+C_{k+1}}{2} 
\]
(namely setting $b=0$). This leads to the formula
\begin{align*}
&(\nabla_{C_{k+1}} + {\textstyle{\frac{1}{2}}} \nabla_{C_k}) \times \nabla_{C_k} \Gamma \\
=& - 2 \int_0^1 (1-a) \omega (a C_k + (1-a) {\textstyle{\frac{C_{k-1} + C_{k+1}}{2}}})\, da \\
& - \int_0^1 a (1-a) ({\textstyle{\frac{C_{k-1}+C_{k-1}-2C_k}{2}}}\cdot \nabla \omega) (C_k + (1-a) {\textstyle{\frac{C_{k+1}+C_{k-1}-2C_k}{2}}})\, da + R^k_{\rm ad} \\
=& - 2\int_0^1 (1-a) \omega (a C_k + (1-a) {\textstyle{\frac{C_{k+1}+C_{k-1}}{2}}})
\, da\\
&\, \, - \int_0^1 a (1-a) \frac{d}{da} \left(\omega (C_k + (1-a) {\textstyle{\frac{C_{k+1}+C_{k-1}-2C_k}{2}}})\right)\, da + R^k_{\rm ad}\\
=& - 2\int_0^1 (1-a) \omega (a C_k + (1-a) {\textstyle{\frac{C_{k+1}+C_{k-1}}{2}}})
\, da\\
&\, \, +\int_0^1 (1-2a) \omega (C_k + (1-a) {\textstyle{\frac{C_{k+1}+C_{k-1}-2C_k}{2}}})\, da + R^k_{\rm ad}\, \\
= & - \int_0^1 \omega (C_k + (1-a) {\textstyle{\frac{C_{k+1}+C_{k-1}-2C_k}{2}}})\, da + R^k_{\rm ad}\, ,
\end{align*}
where $R^k_{\rm ad}$ satisfies \eqref{est:R_ad}.
\end{proof}

\subsection{Velocity and Regularized Biot-Savart Operator}

In Lemma~\ref{lemma:nabla x nabla 1} we proved that the discretized area derivative applied to the circulation $\Gamma[u,C]$ for some $C \in \mathcal L _N \R^3$ produces, up to a small error $R^k_{\rm ad}$, the average of the vorticity $\omega = \nabla \times u$ along the segment $[C_k, \tfrac{1}{2}(C_{k-1}+C_{k+1})]$. 
For convenience we recall the identity:
\begin{equation}\label{eq:final_copy}
	\nabla_{C_{k+1}} \times \nabla_{C_k} \Gamma 
	= - \int_0^1 \omega\!\left(a C_k + (1-a)\tfrac{C_{k-1}+C_{k+1}}{2}\right)\, da 
	+ R^k_{\rm ad}.
\end{equation}

To derive the momentum equation, we need a differential operator that retrieves the velocity field $u$ from $\Gamma$. 
A natural first attempt is to apply the standard Biot--Savart operator
\begin{equation}\label{eq:BS}
	{\rm BS}[\omega](x) := 
	\frac{1}{4\pi} \int_{\R^3} \omega(y) \times \nabla_y \!\left( \frac{1}{|x-y|}\right)\, dy
\end{equation}
to the first integrand in \eqref{eq:final_copy}. 
However, this approach fails, since averaging along the segment produces a logarithmic blow-up. 
More precisely, if we denote
\begin{equation}\label{e:media-segmento}
    \langle \omega \rangle_k(C_k) := \int_0^1 
    \omega\!\left(a C_k + (1-a)\tfrac{C_{k-1}+C_{k+1}}{2}\right)\, da,
\end{equation}
then
\begin{equation}
    {\rm BS}[\langle \omega \rangle_k] 
    = \int_0^1 \frac{1}{a}\, u\!\left(a C_k + (1-a)\tfrac{C_{k-1}+C_{k+1}}{2}\right)\, da,
\end{equation}
which is not integrable at $a=0$.

For this reason, we introduce a regularized version of the Biot--Savart operator, ${\rm BS}_\ell$, depending on a small scale parameter $\ell \ll 1$. The precise definition of the operator is as follows. Consider a cut-off function $\chi \in C_c^\infty(\R^3)$ such that $\chi = 1$ on $B_1(0)$ and $\chi = 0$ outside $B_2(0)$. 
For $\ell \le 1$, we define the operator
\begin{equation}\label{eq:BSN1}
	{\rm BS}_\ell[\omega](x) := 
	\frac{1}{4\pi} \int_{\R^3} \omega(y) \times \nabla_y \!\left( \frac{\chi(\ell(x-y))}{|x-y|}\right)\, dy.
\end{equation}
Notice that ${\rm BS}_\ell$ is well-defined for every $\omega \in L^1_{\mathrm{loc}}(\R^3;\R^3)$. 
Moreover, if $\omega$ is differentiable with locally integrable derivative, we can integrate by parts and rewrite
\begin{equation}\label{eq:BSN2}
	{\rm BS}_\ell[\omega](x) = 
	\frac{1}{4\pi} \int_{\R^3} \frac{\chi(\ell(x-y))}{|x-y|}\, (\nabla \times \omega)(y)\, dy\, .
\end{equation}
The properties of this operator which are relevant for our purposes will be discussed in Appendix ~\ref{sec:BSreg}. 
The main result of this section is the following.

\begin{lemma}\label{lemma:BSlog}
Fix $0 \le \ell, \rho \le 1$ and an integer $K\geq 0$.
     With the above notation, for $0\leq k \leq N-1$ and under the assumption that $u\in C^{K+1} \cap H^k$ we have:
	\begin{equation}\label{e:finale u}
		{\rm BS}_\ell [\langle \omega \rangle_k] 
        = \log(1/\rho)\, u(C_k) + R^k_{\rm BS}
	\end{equation}
    where $R^k_{{\rm BS}} = R^k_{{\rm BS},1} + R^k_{{\rm BS},2} + R^k_{{\rm BS},3}$ (the Biot-Savart error) satisfies
    \begin{equation}
        \begin{split}
            &\|D^K_{C_k} R^k_{{\rm BS},1}\|_{L^\infty} \le \const \ell^{-1}\rho^{K+1} \| D^K \omega \|_{L^\infty}
            \\
            &\sup_{C_{k-1}, C_{k+1}} \| D^K_{C_k} R^k_{{\rm BS},2}\|_{L^2_{C_k}}
            \le \const(K)\ell^{3/2} \log(1/\rho) \|u \|_{H^K},
            \\
            &  
            | R^k_{{\rm BS},3}| (C_{k-1}, C_k, C_{k+1})
    \le \log(1/\rho) (|\triangle C_{k-1}| + |\triangle C_k|) \| Du\|_{L^\infty}.
        \end{split}
    \end{equation}
\end{lemma}

In the application of Lemma \ref{lemma:BSlog}, the parameters $\rho$ and $\ell$ will both depend on the discretization parameter $N$. More precisely, we will choose $\rho \sim \frac{1}{N}$ and $\ell$ as a suitable power of $1/N$.

\begin{remark}\label{r:bound-R_BS}
Before coming to the proof of the lemma, we remark that, by the Sobolev embedding $H^2 (\mathbb R^3) \subset L^\infty (\mathbb R^3)$, the estimates on the various pieces of $R^k_{{\rm BS}}$ imply that 
\[
|R^k_{{\rm BS}}| \leq \const \log (1/\rho) \big((\ell^{-1} \rho + |C_{k+1}-C_k|+|C_k-C_{k-1}|) \|Du\|_{L^\infty} + \ell^{3/2} \|u\|_{H^2}\big)
\]
\end{remark}

\begin{proof}[Proof of Lemma \ref{lemma:BSlog}]
    We split
    \begin{equation}\label{eq:split1}
        \langle \omega \rangle_k(C_k) 
        = 
        \int_0^\rho \omega\!\left(a C_k + (1-a)\tfrac{C_{k-1}+C_{k+1}}{2}\right)\, da
        +\int_\rho^1 \omega\!\left(a C_k + (1-a)\tfrac{C_{k-1}+C_{k+1}}{2}\right)\, da.
    \end{equation}
The first integrand, denoted $I(C_k)$ satisfies the estimate
\begin{equation}
    \| D^K_{C_k} I\|_{L^\infty} 
    \le \frac{\rho^{K+1}}{K+1}\| D^K \omega\|_{L^\infty}, \quad K\ge 0,
\end{equation}
therefore, by Proposition \ref{prop:polygrowth} we have
\begin{equation}
    \| D^K_{C_k} {\rm BS}_\ell[I]\|_{L^\infty}
    \le \const \ell^{-1}  \rho^{K+1}\| D^K \omega\|_{L^\infty}. 
\end{equation}
We set $R^k_{{\rm BS},1}:= {\rm BS}_\ell[I]$.

We now study ${\rm BS}_\ell[II]$, where $II$ is the second integrand in \eqref{eq:split1}. Using that $\omega = \nabla \times u$ and Proposition \ref{prop:BS_Ncurl} we obtain
\begin{equation}\label{eq:split2}
    \begin{split}
        {\rm BS}_\ell[II] & = \int_\rho^1 \frac{1}{a} {\rm BS}_\ell \left[\nabla \times (u(a \cdot + (1-a)\tfrac{C_{k-1}+C_{k+1}}{2})) \right] da
        \\
        &= \int_\rho^1 \frac{1}{a} u(aC_k + (1-a)\tfrac{C_{k-1}+C_{k+1}}{2})\, da + R^k_{{\rm BS},2}
    \end{split}
\end{equation}
where the error term $ R^k_{{\rm BS},2}$ satisfies the estimate
\begin{equation}
    \|  R^k_{{\rm BS},2} \|_{H^K_{C_k}} \le  \int_\rho^1\frac{1}{a} \const(k)\ell^{3/2} \| u \|_{H^k} da
    = \const(K)\ell^{3/2}\log(1/\rho) \|u\|_{H^K}.
\end{equation}

We finally study the main term in \eqref{eq:split2}. We rewrite it as
\begin{equation}
    \int_\rho^1 \frac{1}{a} u(aC_k + (1-a)\tfrac{C_{k-1}+C_{k+1}}{2})\, da
    =
    \log(1/\rho) u(C_k) + R^k_{{\rm BS},3}
\end{equation}
where
\begin{equation}
    R^k_{{\rm BS},3} := \int_\rho^1\frac{1}{a}\left(  u(aC_k + (1-a)\tfrac{C_{k-1}+C_{k+1}}{2}) - u(C_k)\right)da.
\end{equation}
Form the identity
\begin{equation}
    aC_k + (1-a) \frac{C_{k-1} + C_{k+1}}{2} - C_k = (1-a) \frac{\triangle C_k - \triangle C_{k-1}}{2},
\end{equation}
we obtain the uniform bound
\begin{equation}
    |R^k_{{\rm BS},3}|
    \le \log(1/\rho) (|\triangle C_{k-1}| + |\triangle C_k|) \| Du\|_{L^\infty}.
\end{equation}
\end{proof}

\section{Discretized vorticity, velocity, advection, and diffusion operators}
\label{sec:discretized1}

So far the calculus developed in the previous section was always applied to the circulation functional $\Gamma$. 
Here we turn to study functionals of the form
\begin{equation}
	\Psi[u,C] =  \exp \left\lbrace \frac{i \gamma }{\nu} \Gamma[u, C_0, \ldots, C_{N-1}] \right\rbrace \, .
\end{equation}
We eventually apply our computations to $\Psi [u (\cdot, t), C] = \Psi [ u, C, t]$ when $u$ is a (smooth enough) solution of the Navier-Stokes equations in order to derive a discretized version of the momentum equation for $\Psi[u, C, t]$, up to error terms which we will carefully control in terms of $N^{-\alpha}$ for a suitable positive $\alpha$.

\subsection{The Vorticity and Diffusion Operators}
For $0 \le k \le N-1$, we define the \emph{vorticity operator} by
\begin{equation}\label{eq:vor_op}
    \Omega_k := \frac{i\nu}{\gamma}\,\nabla_{C_{k+1}} \times \nabla_{C_k}.
\end{equation}

The terminology is justified as follows: when applied to loop-type functionals of the form
\begin{equation}
	\Psi[u,C] := \exp\!\left\{ \frac{i\gamma}{\nu}\, \Gamma[u, C_0, \ldots, C_{N-1}] \right\},
\end{equation}
one finds that
$\Omega_k \Psi = \omega (C_k)\, \Psi + \text{error}$.
The precise statement is the following.

\begin{lemma}\label{l:exponential-computation-1}
Assume $u\in C^2$ and $0\leq k \leq N-1$. Then
\begin{equation}\label{e:exponential-computation-1}
    \Omega_k \Psi[u, C_0, \ldots, C_{N-1}] = \langle \omega \rangle_k(C_k) \Psi  -  R^k_{\rm ad} \Psi  
    - \frac{i \gamma}{\nu}\Psi \nabla_{C_{k+1}}\Gamma  \times \nabla_{C_k}\Gamma\, ,
\end{equation}
where $\langle \omega\rangle_k$ is given by the formula \eqref{e:media-segmento} and $R^k_{\rm ad}$ is as in Lemma 
\ref{lemma:nabla x nabla 1}. 
\end{lemma}
\begin{proof}
We first use the chain rule to compute
\begin{align}
\Omega_k \Psi &= \frac{i\nu}{2\gamma} \nabla_{C_{k+1}} \times \left( \nabla_{C_k} \Psi \right)
= - \nabla_{C_{k+1}} \times \left( \Psi \nabla_{C_k} \Gamma \right)\nonumber\\
&= - \Psi \nabla_{C_{k+1}} \times \nabla_{C_k} \Gamma 
- \frac{i\gamma}{\nu} \Psi \nabla_{C_{k+1}} \Gamma \times \nabla_{C_k} \Gamma \, ,\label{e:chain}
\end{align}
and then apply Lemma 
\ref{lemma:nabla x nabla 1} to the first term.
\end{proof}

Likewise, we introduce a {\it diffusion operator} which is defined as 
\[
\mathcal{D}_k := 2 \nabla_{C_k} \times (\Omega_k)\, .
\]
Again the heuristic behind the latter definition is that 
\[
\mathcal{D}_k \Psi = (\nabla \times \omega) (C_k) \Psi + \text{error}\, . 
\]
The precise statement is given in the following 

\begin{lemma}\label{l:diffusion}
Assume $u\in C^3$ and $0\leq k \leq N-1$. Then
\begin{equation}\label{e:diffusion}
    \mathcal{D}_k \Psi[u, C_0, \ldots, C_{N-1}] = (\nabla \times \omega) (C_k) \Psi  + R^k_{\rm diff}\, ,
\end{equation}
where $R^k_{\rm diff}$ satisfies the estimate
\begin{equation}\label{e:diffusion-2}
    \begin{split}
      |R^k_{\rm diff}| &\le \const \|\omega \|_{C^2}( 1 + \|\omega\|_{C^2}(1 + L_{k-1,k+1})^2 )(|\triangle C_k| + |\triangle C_{k-1}|)  
      \\& \quad 
      + \const \frac{\gamma}{\nu} \|\omega\|_{C^2}^2|\triangle C_k|( 1 + |\triangle C_{k-1}| + \|\omega \|_{C^2})(|\triangle C_k| + |\triangle C_{k-1}|) 
    \end{split}
\end{equation}
where $L_{p,q}$ is defined in \eqref{eq:L}. 
\end{lemma}

\begin{proof}
From \eqref{e:exponential-computation-1}, it follows that
\begin{equation}\label{eq:split4}
\begin{split}
    \mathcal{D}_k\Psi 
    & = (\nabla \times \omega)(C_k) \Psi
    \\&+ (2 \nabla_{C_k} \times \langle \omega \rangle_k(C_k) - (\nabla \times \omega)(C_k)) \Psi
    \\
    & - 2\nabla_{C_k} \times \left( R^k_{\rm ad}\Psi + \frac{i\gamma}{\nu} \Psi \nabla_{C_{k+1}}\Gamma \times \nabla_{C_k}\Gamma \right).
\end{split}
\end{equation}
The error $R_{\mathrm{diff}}^k$ is defined as the sum of the terms in the last two lines of \eqref{eq:split4}. 
Notice that
\begin{equation}
    \begin{split}
      |2 \nabla_{C_k}& \times \langle \omega \rangle_k(C_k) - (\nabla \times \omega)(C_k)|
      \\&=
      \Big| \int_0^1 2a (\nabla \times \omega(aC_k + (1-a)\tfrac{C_{k-1} + C_{k+1}}{2}) - \nabla\times \omega (C_k)) da \Big|
      \\& \le 
      \|D^2\omega\|_{L^\infty}(|C_{k+1} - C_k| + |C_k - C_{k-1}|)
    \end{split}
\end{equation}
A straightforward calculation, combined with \eqref{eq:nablaGamma_bound} and Lemma~\ref{lemma:nabla x nabla 1}, shows that
\begin{equation}
\begin{split}
    | \nabla_{C_k} \times (R_{\rm ad}^k \Psi)|
    & \le \frac{\gamma}{\nu} |\nabla_{C_k} \Gamma| |R^k_{\rm ad}| + |D_{C_k} R^k_{\rm ad}| 
    \\&
    \le \const  \frac{\gamma}{\nu}\| D \omega\|_{L^\infty}^2|C_{k+1} - C_{k-1}|^2(1 + |C_k - C_{k-1}|)
    \\& \qquad  + \const \| D^2 \omega \|_{L^\infty} |C_{k+1} - C_{k-1}|.
\end{split}
\end{equation}
Using again \eqref{eq:nablaGamma_bound}, we conclude that
\begin{equation}
    \begin{split}
      |\nabla_{C_k}& \times (\Psi \nabla_{C_{k+1}} \Gamma \times \nabla_{C_k} \Gamma)|
      \\& \le \frac{\gamma}{\nu} |\nabla_{C_k}\Gamma|^2 |\nabla_{C_{k+1}}\Gamma| 
       + |D_{C_k}\nabla_{C_k}\Gamma||\nabla_{C_{k+1}}\Gamma| + |D_{C_{k}}\nabla_{C_{k+1}}\Gamma||\nabla_{C_k}\Gamma|
       \\&
        \le \const \| \omega \|_{C^2}^2 (1 + L_{k-1,k+1})
        \left( \frac{\gamma}{\nu}\|\omega\|_{C^1} |C_{k+1} - C_k| + 1 + L_{k-1,k+1}\right)|C_{k+1} - C_{k-1}|.
    \end{split}
\end{equation}
\end{proof}

\subsection{The Velocity Operator}
\label{subsec:vel_op}

For $0 \le k \le N-1$, and for parameters $0 < \ell, \rho < 1$ (to be chosen later in terms of $N$), we define the \emph{velocity operator} by
\begin{equation}\label{eq:vel_op}
    U_k := \frac{1}{\log(1/\rho)}\, {\rm BS}_\ell \circ \Omega_k,
\end{equation}
where the regularized Biot-Savart operator ${\rm BS}_\ell$ (see \eqref{eq:BSN1}) acts on the variable $C_k$. 
Recall that $\Omega_k$ denotes the vorticity operator from \eqref{eq:vor_op}. 
For brevity, we omit the explicit dependence of $U_k$ on the parameters $\rho$ and $\ell$.

As in the vorticity case, the terminology is justified by the fact that
\[
    U_k \Psi = u(C_k)\,\Psi + \text{error}.
\]
The precise statement is given below.

\begin{lemma}\label{lemma:vel_op}
    Let $0\le k \le N-1$ and $0< \ell, \rho<1$. For every $u\in C^2$ we have
    \begin{equation}
        U_k \Psi[C_0, \ldots, C_{N-1}]
        =
        u(C_k) \Psi + \frac{R^k_{\rm BS}}{\log(1/\rho)} \Psi + R^k_{\rm vel}
    \end{equation}
    where $R^k_{\rm BS}$ is defined in Lemma \ref{lemma:BSlog} and the velocity error $R^k_{\rm vel}:= R^k_{{\rm vel},1} + R^k_{{\rm vel},2}$ satisfies the estimates:
    \begin{align}
        & |R^k_{{\rm vel},1}| \le \const \ell^{-1} |C_{k+1}-C_{k-1}| \|D\omega\|_{L^\infty} + \const \frac{\gamma}{\nu} \ell^{-2} |C_{k+1}- C_{k-1}||C_{k+2}-C_k| \|D\omega\|_{L^\infty}^2\, .
        \\
        & |R^k_{{\rm vel},2}| \le \const \frac{\gamma}{\nu} \ell^{-1} \log (1/\rho) |C_{k+1}-C_{k-1}| \|D \omega\|_{L^\infty}\, .
    \end{align}
\end{lemma}

\begin{proof}
We apply $\tfrac{1}{\log(1/\rho)}{\rm BS}_\ell$ with respect to the variable $C_k$ to both sides of \eqref{e:exponential-computation-1}, and set
\begin{equation}
    R^k_{{\rm vel},1} := -\frac{1}{\log(1/\rho)}\, {\rm BS}_\ell\left[
        R^k_{\rm ad}\Psi 
        + \frac{i\gamma}{\nu}\, \Psi \nabla_{C_{k+1}}\Gamma \times \nabla_{C_k}\Gamma
    \right].
\end{equation}
We thus obtain the identity
\begin{equation}\label{eq:split3}
    U_k \Psi[C_0,\ldots,C_{N-1}] 
    = \frac{1}{\log(1/\rho)}\, {\rm BS}_\ell[\langle \omega \rangle_k \Psi] 
    + R^k_{\rm vel,1}\, .
\end{equation}
To study the first term in \eqref{eq:split3} we introduce the notation
\begin{equation}
    \triangle \Psi(x) 
    := \Psi[C_0, \ldots, x , \ldots, C_{N-1}]
       - \Psi[C_0, \ldots, C_k, \ldots, C_{N-1}],
    \quad x \in \R^3,
\end{equation}
for the discrete difference of $\Psi$ in the $k$-th variable. We then employ Lemma~\ref{lemma:BSlog} to write
\begin{equation}
\begin{split}
   {\rm BS}_\ell[\langle \omega \rangle_k \Psi]
   &= {\rm BS}_\ell[\langle \omega \rangle_k] \Psi 
      + {\rm BS}_\ell[\langle \omega \rangle_k \, \triangle \Psi] \\
   &= \log(1/\rho)\, u(C_k)\,\Psi 
      + R^k_{\rm BS}\,\Psi 
      + {\rm BS}_\ell[\langle \omega \rangle_k \, \triangle \Psi]\, ,
\end{split}
\end{equation}
If we therefore define
\begin{equation}
    R^k_{{\rm vel},2} : = \frac{1}{\log(1/\rho)} {\rm BS}_\ell[\langle \omega \rangle_k \, \triangle \Psi]
\end{equation}
and 
\[R^k_{\rm vel}:= R^k_{{\rm vel},1} + R^k_{{\rm vel},2}\, 
\]
we reach the claimed identity for $U_k \Psi$ and to complete the proof we need to estimate the two error terms $R^k_{{\rm vel},i}$.

\medskip
\noindent
{\it Estimate on $R^k_{{\rm vel},1}$.}
Using Lemma \ref{l:nablaGamma_bound} and Lemma \ref{lemma:nabla x nabla 1} we get
\begin{equation}
\begin{split}
    &\left|R^k_{\rm ad}\Psi + \frac{i\gamma}{\nu}\, \Psi \nabla_{C_{k+1}} \Gamma \times \nabla_{C_k}\Gamma \right| 
    \le 
    \const |C_{k+1} - C_{k-1}| \| D \omega \|_{L^\infty}
    \\& \qquad + \const \frac{\gamma}{\nu} |C_{k+1} - C_{k-1}| |C_{k+2}-C_k|\|D\omega\|_{L^\infty}^2 \, . 
\end{split}
\end{equation}
Hence, by Proposition \ref{prop:polygrowth} we have
\begin{equation}
      | R^k_{{\rm vel},1} |
    \le
    \const \ell^{-1} |C_{k+1}-C_{k-1}| \|D\omega\|_{L^\infty} + \const \frac{\gamma}{\nu}\ell^{-2} |C_{k+1}- C_{k-1}||C_{k+2}-C_k|
    \|D\omega\|_{L^\infty}^2\, .
\end{equation}

\medskip
\noindent
{\it Estimate on $R^k_{{\rm vel},2}$.}
By the mean value theorem, there exists $s\in [0,1]$ such that
\begin{equation}
    \triangle \Psi(x) = \frac{i\gamma}{\nu} (x-C_k) \cdot \nabla_{C_k}\Gamma[C_0, \ldots, s C_k + (1-s) x, \ldots, C_{N-1}]\, \Psi.
\end{equation}
By Lemma \ref{l:nablaGamma_bound}, we deduce
\begin{equation}
\begin{split}
    |\triangle \Psi(x)| 
    &\le \const \frac{\gamma}{\nu} |x-C_k||C_{k+1} - C_{k-1}|(1 + |sC_k + (1-s)x - C_{k-1}|) \| D\omega \|_{L^\infty}
    \\& \le \const \frac{\gamma}{\nu} |x-C_k||C_{k+1} - C_{k-1}|(|x - C_k| + |C_k - C_{k-1}|) \| D\omega \|_{L^\infty}.
\end{split}    
\end{equation}
Finally, we apply Proposition \ref{prop:polygrowth} to deduce
\begin{equation}
    | R^k_{{\rm vel},2} |
    \le
    \const \frac{\gamma}{\nu} \frac{\ell^{-2}}{\log(1/\rho)}|C_{k+1} - C_{k-1}|(1 + |\triangle C_{k-1}|) \| D \omega \|_{L^\infty}.
\end{equation}
The estimates for the derivatives follow analogously.    
\end{proof}

\subsection{The Advection Operator}
For $0 \le k \le N-1$, and for parameters $0 < \ell, \rho < 1$ (to be chosen later in terms of $N$), 
to reconstruct the nonlinear term $\omega \times u$ appearing in the Navier--Stokes equations \eqref{eq:NS2},
we define the \emph{advection operator} by
\begin{equation}
    \Omega_{k+3} \times U_k
    := 
    \varepsilon_{\alpha \ell j}\, e_\alpha \, \Omega_{k+3}^\ell \circ U_k^j,
\end{equation}
where $\Omega_{k+3}^\ell$ and $U_k^j$ denote the $\ell$th and $j$th components of 
$\Omega_{k+3}$ and $U_k$, respectively, and $(e_\alpha)_{\alpha=1}^3$ is the canonical basis of $\R^3$.

\begin{lemma}\label{l:advect-operator}
Let $0\le k \le N-1$ and $0< \ell, \rho<1$. For every $u\in C^2$ we have
\begin{equation}
        (\Omega_{k+3} \times U_k) \Psi[C_0, \ldots, C_{N-1}]
        =
        (\omega \times u)(C_k)\Psi + R^k_{\rm adv}
\end{equation}
where, the advection error $R^k_{\rm adv}$ satisfies
\begin{align*}
|R^k_{\rm adv}| \leq & \const L_{k-1,k+3} \|u\|_{L^\infty} \|D\omega\|_{L^\infty}
\Big(1+ L_{k-1,k+3}\frac{\gamma}{\nu} \|D\omega\|_{L^\infty}\Big)\\
& + \const (1+L_{k-1,k+3}) \big((\ell^{-1} \rho + L_{k-1,k+3}) \|Du\|_{L^\infty} + \ell^{3/2} \|u\|_{H^2}\big) \|\omega\|_{C^1}\\
& + \const \frac{\gamma}{\nu} (1+L_{k-1,k+3}) \ell^{-1} \log (1/\rho) L_{k-1,k+3} \|D\omega\|_{L^\infty} (1+ \ell^{-1} L_{k-1,k+3}\|D\omega\|_{L^\infty})
\|\omega\|_{C^1}.\\
\end{align*}
\end{lemma}

\begin{proof}
    By Lemma \ref{lemma:vel_op}, we can write 
\begin{equation}
        ((\Omega_{k+3} \times U_k) \Psi )_\alpha
        =
        \eps_{\alpha i j}\Omega^i_{k+3} \left(u_j(C_k) \Psi + \frac{(R_{\rm BS})_j}{\log(1/\rho)} \Psi + (R_{\rm vel})_j\right)
    \end{equation}
Since  $\Omega_{k+3}$ is a differential operator acting only on variables $C_{k+3}$ and $C_{k+4}$, by Lemma~\ref{lemma:nabla x nabla 1} we obtain
\begin{equation}
    \Omega_{k+3}(u_j(C_k) \Psi)
    = \langle \omega \rangle_{k+3}(C_{k+3}) u_j(C_k) \Psi - R_{\rm ad}^{k+3} u_j(C_k) \Psi 
    - \frac{i\gamma}{\nu} u_j(C_k) \Psi \nabla_{C_{k+4}\Gamma}\times \nabla_{C_{k+3}}\Gamma
\end{equation}
We define
\begin{equation}
    \begin{split}
        (R^k_{\rm adv})_\alpha 
        &:= 
        \eps_{\alpha \ell j} (\langle \omega \rangle_{k+3}(C_{k+3}) - \omega(C_k))_\ell u_j(C_k)) \Psi
        \\& - \eps_{\alpha \ell j}( (R_{\rm ad}^{k+3})_\ell + \frac{i\gamma}{\nu}  (\nabla_{C_{k+4}\Gamma}\times \nabla_{C_{k+3}}\Gamma)_\ell) u_j(C_k) \Psi 
        \\&
        + \eps_{\alpha \ell j} \Omega_{k+3}^\ell \left( \frac{(R^k_{\rm BS})_j}{\log(1/\rho)} \Psi + (R_{\rm vel})_j\right).
    \end{split}
\end{equation}
First of all, a simple computations shows that 
\[
|\langle \omega \rangle_{k+2} (C_{k+3}) - \omega (C_{k+3})|
\leq \const \|D\omega\|_{L^\infty} (|C_{k+4}-C_{k+3}|+ |C_{k+3}-C_{k+2}|)\, ,
\]
while obviously 
\[
|\omega (C_{k+3})-\omega (C_k)| \leq \|D\omega\|_{L^\infty} \sum_{j=k}^{k+2} |C_{j+1}-C_j|\, .
\]
As for $R^{k+3}_{\rm ad}$ we recall Lemma \ref{lemma:nabla x nabla 1} to get 
\[
|R^{k+3}_{\rm ad}| \leq \const \|D\omega\|_{L^\infty} |C_{k+4}-C_{k+2}|
\leq \const \|D\omega\|_{L^\infty} (|C_{k+4}-C_{k+3}|+|C_{k+3}-C_{k+2}|) \, .
\]
Next, recalling Lemma \ref{l:nablaGamma_bound}, we have 
\[
|\nabla_{C_{k+4}\Gamma}\times \nabla_{C_{k+3}}\Gamma| 
\leq \const \|D\omega\|_{L^\infty}^2 |C_{k+4} - C_{k+2}||C_{k+1}-C_{k-1}|
\]
We now come to the last two terms. First of all recall that the error terms 
$R^k_{{\rm BS}}$ only depend on the variables $C_{k+1}, C_k, C_{k-1}$ while the differential operator $\Omega_{k+3}$ involves derivatives in $C_{k+3}$ and $C_{k+4}$. Therefore we can write 
\[
I := \Omega_{k+3}^\ell \left( \frac{(R^k_{\rm BS})_j}{\log(1/\rho)} \Psi\right)
= \frac{(R^k_{\rm BS})_j}{\log(1/\rho)} \Omega_{k+3}^\ell \Psi\, .
\]
Now, recalling the estimates in Lemma \ref{lemma:BSlog} and Remark \ref{r:bound-R_BS} we get
\[
\left|\frac{(R^k_{\rm BS})_j}{\log(1/\rho)}\right|  \leq \const \big((\ell^{-1} \rho + |C_{k+1}-C_k|+|C_k-C_{k-1}|) \|Du\|_{L^\infty} + \ell^{3/2} \|u\|_{H^2}\big)
\]
while of course 
\[
|\Omega^{k+3} \Psi| \leq \const \|\omega\|_{L^\infty} + |R^{k+3}_{\rm ad}|
\leq C \big(\|\omega\|_{L^\infty} + (|C_{k+4}-C_{k+3}|+|C_{k+3}-C_{k+2}) \|D\omega\|_{L^\infty}\big)
\]
In order to estimate the final term, we need to inspect the exact form of the error terms forming $R^k_{\rm vel}$. First of all we recall the formula
\[
R^k_{{\rm vel},1} := -\frac{1}{\log(1/\rho)}\, {\rm BS}_\ell\left[
        R^k_{\rm ad}\Psi 
        + \frac{i\gamma}{\nu}\, \Psi \nabla_{C_{k+1}}\Gamma \times \nabla_{C_k}\Gamma
    \right]
\]
We again observe that, due to the ``shift'' in the variables appearing in the differential operator $\Omega^{k+3}$ we can write the identity
\[
\Omega^{k+3} R^k_{{\rm vel}, 1}
= -\frac{1}{\log(1/\rho)}\, {\rm BS}_\ell\left[
        R^k_{\rm ad}\Psi 
        + \frac{i\gamma}{\nu}\, \Psi \nabla_{C_{k+1}}\Gamma \times \nabla_{C_k}\Gamma
    \right] \Psi^{-1} (\Omega^{k+3} \Psi))\, .
\]
Owing to the fact that $|\Psi|=1$, we can then estimate 
\begin{align*}
|\Omega^{k+3} R^k_{{\rm vel}, 1}| &\leq |\Omega^{k+3} \Psi| |R^k_{{\rm vel}, 1}|\\
&\leq \const (1+L_{k-1,k+3}) \ell^{-1} L_{k-1,k+3} \|D\omega\|_{L^\infty} (1+ \ell^{-1} \frac{\gamma}{\nu} L_{k-1,k+3}\|D\omega\|_{L^\infty})
\|\omega\|_{C^1}
\end{align*}
We next observe that the same structure holds for the error term $R^k_{{\rm vel}, 2}$, namely that 
\[
\Omega^{k+3} R^k_{{\rm vel}, 2} = R^k_{{\rm vel}, 2} \Psi^{-1} \Omega^{k+3} \Psi\, ,
\]
leading therefore to the estimate
\begin{align*}
|\Omega^{k+3} R^k_{{\rm vel}, 2}| &\leq \const \frac{\gamma}{\nu} \ell^{-1} \log (1/\rho) L_{k-1,k+3} \|D\omega\|_{L^\infty}\, .
\end{align*}
\end{proof}

\section{Migdal's Original Discretization}\label{s:Sasha}

As explained in the introduction, the discrete vorticity and velocity operators introduced in Section~\ref{sec:discretized1} are a variant of those originally proposed by Migdal in \cite{Migdal23,Migdal24,Migdal25,Migdalnuovo}. 
In this section we recall Migdal’s original operators and highlight the main technical differences with respect to our discretization.

\subsection{New Variables} 
By translation invariance, we assume $C_0=0$. With this choice, we can regard $C_1, \ldots, C_{N-1}$ as independent variables in $(\mathbb{R}^3)^{N-1}$, and introduce the increments
\[
s_k := \triangle C_k, \qquad k = 0,\ldots,N-2,
\]
so that we may pass from the variables $C_1, \ldots, C_{N-1}$ to $s_0, \ldots, s_{N-2}$.
Recall that $C_N = C_0 = 0$.
The change of variables is given explicitly by
\[
C_1 = s_0, 
\qquad 
C_k = s_0 + \cdots + s_{k-1}, \quad k = 1,\ldots, N-1,
\]
which immediately yields the identity
\begin{equation}\label{eq:nablas_k}
    \nabla_{s_k} 
    = \nabla_{C_{k+1}} + \cdots + \nabla_{C_{N-1}},
    \qquad k=0,\ldots, N-2.
\end{equation}

\subsection{Migdal's Discretized Area Derivative}

In his works, Migdal proposes the following discretization of the area derivative:
\begin{equation}\label{eq:AreDev Migdal}
    \frac{1}{2}\bigl(\nabla_{s_{k}} + \nabla_{s_{k-1}}\bigr) \times \nabla_{C_k}\Gamma .
\end{equation}
Notice that, by \eqref{eq:nablas_k}, this can be expressed in terms of our discretized area derivative as
\begin{equation}
    \frac{1}{2}\bigl(\nabla_{s_{k}} + \nabla_{s_{k-1}}\bigr) \times \nabla_{C_k} 
    =
    \nabla_{C_{k+1}} \times \nabla_{C_k}
    +
    \left(\sum_{\ell=k+2}^{N-1}\nabla_{C_\ell}\right) \times \nabla_{C_k}.
\end{equation}
From the explicit expression for $\nabla_{C_k}$ obtained in \eqref{e:formulozza}, 
it follows that Migdal’s discretized area derivative satisfies the same property as in Lemma~\ref{lemma:nabla x nabla 1}.

\begin{lemma}\label{lemma:nabla x nabla 2}
	With the above notation, for $2\leq k \leq N-2$ and any integer $K\geq 0$, if $u\in C^{K+2}$ we then have:
	\begin{equation}
		\frac{1}{2}(\nabla_{s_{k}} + \nabla_{s_{k-1}}) \times \nabla_{C_k} \Gamma 
		= - \int_0^1 \omega (aC_k + (1-a) {\textstyle{\frac{C_{k-1} + C_{k+1}}{2}}})\, da + R_{\rm ad}^k\, ,
	\end{equation}
    where $R_{\rm ad}$ (the area derivative error) satisfies the estimate 
    \begin{equation}
    |D^K_{C_k} R_{\rm ad}^k| (C_1, \ldots , C_{N-1}) \leq C(K) \|D^{K+1} \omega\|_{L^\infty} |C_{k+1} - C_{k-1}|\, .
    \end{equation}
\end{lemma}

\subsection{Migdal's Vorticity Operator}

For $2 \le k \le N-2$, Migdal's vorticity operator is defined by
\begin{equation}\label{eq:vor_op Migdal}
    \Omega_k^M := \frac{i\nu}{\gamma}\,\frac{1}{2}\bigl(\nabla_{s_k} + \nabla_{s_{k-1}}\bigr) \times \nabla_{C_k}.
\end{equation}
Using the properties of Migdal's area derivative \eqref{eq:AreDev Migdal}, 
and arguing as in Lemma~\ref{l:exponential-computation-1}, one obtains the following result:

\begin{lemma}\label{l:exponential-computation-2}
Assume $u \in C^2$ and $2 \le k \le N-2$. Then
\begin{equation}
    \Omega_k^M \Psi[u, C_1, \ldots, C_{N-1}]
    = \langle \omega \rangle_k(C_k)\, \Psi
      - R_{\rm ad}^k\, \Psi
      - \frac{i\gamma}{\nu}\, \Psi \Big( \sum_{j=k+1}^{N-1} \nabla_{C_j}\Gamma \Big) \times \nabla_{C_k}\Gamma ,
\end{equation}
where $\langle \omega \rangle_k$ is given by \eqref{e:media-segmento}, and $R_{\rm ad}$ is as in Lemma~\ref{lemma:nabla x nabla 2}.
\end{lemma}

\subsection{Migdal's Velocity Operator}
For $2 \le k \le N-2$, and parameters $0 < \ell, \rho < 1$, the velocity operator associated with Migdal's vorticity operator is defined by
\begin{equation}\label{eq:vel_op Migdal}
    U_k^M := \frac{1}{\log(1/\rho)}\, {\rm BS}_\ell \circ \Omega_k^M,
\end{equation}
where the regularized Biot-Savart operator ${\rm BS}_\ell$ (see \eqref{eq:BSN1}) acts on the variable $C_k$.  

We obtain the following expansion, whose proof follows verbatim from that of Lemma~\ref{lemma:vel_op}, 
the only minor difference being that the first error term involves a sum of derivatives of the circulation:
\begin{equation}
    R_{{\rm vel},1}^k := -\frac{1}{\log(1/\rho)}\, {\rm BS}_\ell\left[
        R_{\rm ad}^k\Psi 
        + \frac{i\gamma}{\nu}\, \Psi \Big( \sum_{j=k+1}^{N-1} \nabla_{C_j}\Gamma \Big) \times \nabla_{C_k}\Gamma
    \right].
\end{equation}
Compared with the error term in Lemma~\ref{lemma:vel_op}, here additional contributions appear:
\begin{equation}
    {\rm BS}_\ell\left[
        \frac{i\gamma}{\nu}\, \Psi \Big( \sum_{j=k+2}^{N-1} \nabla_{C_j}\Gamma \Big) \times \nabla_{C_k}\Gamma
    \right],
\end{equation}
which, however, do not pose a serious difficulty thanks to the smallness of $\nabla_{C_k}\Gamma$.

\begin{lemma}\label{lemma:vel_op Migdal}
    For any $2\le k \le N-2$, $0< \ell, \rho<1$ and $u\in C^2$, we have
    \begin{equation}
        U_k^M \Psi[C_1, \ldots, C_{N-1}]
        =
        u(C_k) \Psi + \frac{R_{\rm BS}^k}{\log(1/\rho)} \Psi + R_{\rm vel}^k
    \end{equation}
    where $R_{\rm BS}^k$ is defined in Lemma \ref{lemma:BSlog} and the velocity error $R_{\rm vel}^k:= R_{{\rm vel},1}^k + R_{{\rm vel},2}^k$ satisfies the estimates:
      \begin{align}
        & |R^k_{{\rm vel},1}| \le \const \ell^{-1} |C_{k+1}-C_{k-1}| \|D\omega\|_{L^\infty} 
        \\& \quad + \const \frac{\gamma}{\nu} \ell^{-2} |C_{k+1}- C_{k-1}|(1 + |\triangle C_k|) (\sum_{j=k+2}^{N-2} |C_{j-1} - C_{j+1}|(1 + |\triangle C_{j-1}|) \|D\omega\|_{L^\infty}^2\, .
        \\
        & |R^k_{{\rm vel},2}| \le \const \frac{\gamma}{\nu} \ell^{-1} \log (1/\rho) |C_{k+1}-C_{k-1}| \|D \omega\|_{L^\infty}\, .
    \end{align}
\end{lemma}

\subsection{Migdal's Diffusion Operator}
For $2 \le k \le N-2$, starting from Migdal's vorticity operator we define the diffusion operator by
\[
\mathcal{D}_k^M := 2 \nabla_{C_k} \times (\Omega_k^M ).
\]
Arguing as in Lemma~\ref{l:diffusion}, one obtains the same estimates.

\begin{lemma}\label{l:diffusion migdal}
Assume $u\in C^3$ and $0\leq k \leq N-1$. Then
\begin{equation}\label{e:diffusion migdal}
    \mathcal{D}_k^M \Psi[u, C_0, \ldots, C_{N-1}] = (\nabla \times \omega) (C_k) \Psi  + R^k_{\rm diff}\, ,
\end{equation}
where $R^k_{\rm diff}$ satisfies the estimate
\begin{equation}\label{e:diffusion-2.1}
    \begin{split}
      |R^k_{\rm diff}| &\le \const \|\omega \|_{C^2}( 1 + \|\omega\|_{C^2}(1 + L_{k-1,k+1})^2 )(|\triangle C_k| + |\triangle C_{k-1}|)  
      \\& \quad 
      + \const \frac{\gamma}{\nu} \|\omega\|_{C^2}^2|\triangle C_k|( 1 + |\triangle C_{k-1}| + \|\omega \|_{C^2})(|\triangle C_k| + |\triangle C_{k-1}|) 
    \end{split}
\end{equation}
where $L_{p,q}$ is defined in \eqref{eq:L}. 
\end{lemma}

\subsection{Migdal's Advection Operator} 
\label{s:advection}

Using Migdal's discretization, we are not able to establish the analogue of Theorem~\ref{t:loop-equation-discretized}, 
since the advection operator $\Omega_{k+3}^M \times U_k^M$ generates error terms that we are currently unable to control.  
Equivalently, we cannot prove a counterpart of Lemma~\ref{l:advect-operator}.  
The strongest statement we can show at present is a weaker version, in which the additional error terms are only bounded by order $1/\log(N)$: 
\begin{equation}
    (\Omega_{k+3}^M \times U_k^M)\,\Psi[C_0, \ldots, C_{N-1}]
    = (\omega \times u)(C_k)\,\Psi + R^k_{\rm adv} + R_{\rm bad},
\end{equation}
where the remainder $R_{\rm bad}$ satisfies
\[
    |R_{\rm bad}| \le \frac{C}{\log(N)}.
\]

Moreover, in Migdal’s discretization there is no clear benefit in shifting the indices of the advection operator: 
since $\Omega_{k+3}^M$ acts on all variables $C_1,\ldots,C_{N-1}$, most of the cancellations used in the proof of Lemma~\ref{l:advect-operator} are lost.  
In his work, Migdal instead employs the operator $\Omega_k^M \times U_k^M$, and claims that the Navier--Stokes equations implies that
\begin{equation}
\left(\partial_t  
- \frac{i\gamma}{\nu}\sum_{k=0}^{N-1} \triangle C_k \cdot 
   \bigl(\Omega_k^M \times U_k^M - \nu \mathcal{D}_k^M\bigr)\right) 
\Psi = R',
\end{equation}
for some small error term $R'$ (compare with the formulation in Theorem~\ref{t:loop-equation-discretized}).  
At present, however, we are not able to verify this claim.

\section{The discretized loop equation}

In this section we finally derive a discretized version of Migdal's loop equation. We split the section in three parts. In the first one we give a precise statement about the validity of the discretized version in ``physical space''. In the second part we see what happens when we apply the Fourier transform and we pass to the discretized momentum. In the final section we compare our conclusions to Migdal's and more specifically to the equations leading to his ``Euler ensemble''.

\medskip
For simplicity, throughout this section we restrict to loops $C \in \mathcal{L}_N \R^3$ satisfying
\begin{equation}\label{eq:cond M/N}
    \max_k |C_{k+1}-C_k| \le \frac{M}{N},
\end{equation}
for some parameter $M \ge 1$.

\subsection{Discretized Equation in Physical Variables} 
The main conclusion can be summarized in the following statement. 

\begin{theorem}\label{t:loop-equation-discretized}
Fix an integer $N \ge 2$ and parameters $M \ge 1$, $\alpha \in (0,1)$. 
Set $\ell = N^\alpha$ and $\rho = N$ in the definition of the velocity operator \eqref{eq:vel_op}. 
Assume that 
\[
u \in C^1(\R^3 \times [0,T]) \cap C\big([0,T], H^2 \cap C^3\big)
\]
is a classical solution of the incompressible Navier-Stokes equations \eqref{NS}. 
Let $C \in \mathcal{L}_N \R^3$ be a polygonal loop satisfying \eqref{eq:cond M/N}, and define
\begin{equation}
\Psi[u,C,t] 
= \exp\!\left\{\frac{i\gamma}{\nu} \int_0^1 u(C(\theta),t) \cdot C'(\theta)\, d\theta\right\}.
\end{equation}
Then
\begin{equation}\label{eq:Psi1}
\left(\partial_t  
- \frac{i\gamma}{\nu}\sum_{k=0}^{N-1} \triangle C_k \cdot \big(\Omega_{k+3} \times U_k - \nu \mathcal{D}_k\big)\right) 
\Psi
= R_{\mathrm{loop}},
\end{equation}
where the remainder satisfies
\begin{equation}\label{e:loop-error}
|R_{\mathrm{loop}}| 
\le \const(M,\nu,\gamma, \alpha, \|u\|_{C^3}, \|u\|_{H^2})\,\Big(N^{\alpha-1}\log N + N^{-3\alpha/2}\Big).
\end{equation}
\end{theorem}

\begin{proof} 
We first use \eqref{e:loop-derivative} to get 
\[
\partial_t \Psi = \frac{i\gamma}{\nu} \Psi \int_0^1 \big[((\omega\times u) - \nu \nabla \times \omega) (C(\theta), t) \big]
\cdot C' (\theta)\, d\theta\, ,
\]
where the pressure term vanishes since its circulation is zero.
We now split the integral as a sum over the segments $[C_k,C_{k+1}]$ and observe that we can replace it therefore with the following Riemann sum (up to an error) 
\[
\partial_t \Psi = \frac{i\gamma}{\nu} \sum_{k=0}^{N-1} (C_{k+1}-C_k) \cdot ((\omega\times u) (C_k) -\nu (\nabla\times \omega) (C_k)) \Psi + R_{{\rm riem}}\, ,
\]
where we can estimate 
\[
|R_{{\rm Riem}}| \leq \const \frac{\gamma}{\nu} M N^{-1} (\|D^2 \omega\|_{L^\infty}+ \|Du\|_{L^\infty}^2 + \|Du\|_{L^\infty} \|D\omega\|_{L^\infty})\, .
\]
We next apply Lemma \ref{l:diffusion} and Lemma \ref{l:advect-operator} with parameters $\rho=1/N$ and $\ell = N^{-\alpha}$
to replace each instance of $(\omega\times u) (C_k) \Psi$ with $(\Omega_{k+3}\times U_k) \Psi$ and each instance of $-\nu (\nabla \times \omega) (C_k) \Psi$ with $-\nu \mathcal{D}_k \Psi$, obtaining
the sought identity \eqref{eq:Psi1}, with
\begin{equation}
    R_{\rm loop} = R_{\rm riem} + \frac{i \gamma}{\nu}\sum_{k=0}^{N-1}(C_{k+1} - C_{k}) \cdot (R^k_{\rm adv} - \nu R^k_{\rm diff}).
\end{equation}
We obtain the estimate:
\begin{equation}
    \begin{split}
        |R_{\rm loop}| 
        &\le |R_{\rm riem}| + \frac{\gamma}{\nu} M N^{-1}\sum_{k=0}^{N-1}(|R^k_{\rm adv}| + \nu |R^k_{\rm diff}|)
        \\&
        \le |R_{\rm riem}|
        + \frac{\gamma}{\nu} M N^{-1}\sum_{k=0}^{N-1}
        \const(\|u\|_{C^3}, \|u \|_{H^2}, \frac{\gamma}{\nu})
        \Big(N^{\alpha-1}\log N + N^{-3\alpha/2} + \nu \frac{M}{N}\Big)
        \\& \le \const(M,\nu,\gamma, \alpha, \|u\|_{C^3}, \|u\|_{H^2})\,\Big(N^{\alpha-1}\log N + N^{-3\alpha/2}\Big).
    \end{split}
\end{equation}
\end{proof}

\subsection{Discretized Equation in Momentum Variables} 
Recall next that the aim is to understand the action of the operators when $\Psi$ is expressed in ``momentum variables'', namely it is of the form
\begin{equation}\label{e:Psi-in-P}
\Psi [P, C, t] = \exp \left\{\frac{i\gamma}{\nu} \sum_{k=0}^{N-1} P_k (t) \cdot \triangle C_k \right\}
= \exp \left\{ - \frac{i\gamma}{\nu} \sum_{k=0}^{N-1} C_k \cdot \triangle P_{k-1}(t)\right\}\, .
\end{equation}
where
\[
\triangle P_k := P_{k+1}-P_k\, .
\]

\begin{remark}
In \eqref{e:Psi-in-P} we make the unrealistic assumption that $\Psi$ is deterministic in the momentum variables, that is, $\beta$ is a delta measure concentrated at $P(t)$. However, this simplification does not affect our argument, since it is pathwise.
\end{remark}

We use Proposition \ref{prop:simple_waves} to get the following

\begin{theorem}\label{t:loop-equation-discretized-2}
Fix an integer $N \ge 2$ and parameters $M, \Lambda \ge 1$, $\alpha \in (0,1)$. 
Set $\ell = N^\alpha$ and $\rho = N$ in the definition of the velocity operator \eqref{eq:vel_op}. 
Assume that 
\begin{equation}\label{e:tecnica}
{\rm Im}\, \triangle P_k = 0 \quad\mbox{and} \quad   \frac{1}{\Lambda} \le |\triangle P_k| \le \Lambda,
   \quad k=0,\ldots, N-1
\end{equation}
Let $C \in \mathcal{L}_N \R^3$ be a polygonal loop satisfying \eqref{eq:cond M/N}.
Then 
\begin{align}
& \left(\partial_t  
- \frac{i\gamma}{\nu}\sum_{k=0}^{N-1} \triangle C_k \cdot \big(\Omega_{k+3} \times U_k - \nu \mathcal{D}_k\big)\right) 
\Psi
\nonumber\\
= & \frac{i \gamma}{\nu} \Psi \sum_{k=0}^{N-1} \triangle C_k \cdot \left(\frac{d}{dt}P_k(t) - E_k(t)\right) + R_{\rm mo}\label{e:altra-formula}
\end{align}
where
\begin{equation}\label{eq:E_k}
\begin{split}
    E_k := \frac{1}{\log(N)}&\frac{i\gamma}{\nu} (\triangle P_{k+3} \times \triangle P_{k+2}) \times \left( \triangle P_k - \frac{\triangle P_{k-1}\, (\triangle P_{k-1}\cdot \triangle P_k)}{|\triangle P_{k-1}|^2}\right)
        \\& + \frac{2\gamma^2}{\nu} (\triangle P_k\, |\triangle P_{k-1}|^2 - \triangle P_{k-1}\, (\triangle P_{k-1}\cdot \triangle P_k)).
\end{split}    
\end{equation}
and the momentum error $R_{\rm mo}$ satisfies the following estimate:
\[
|R_{\rm mo}| \leq \const(m,\gamma,\nu,\Lambda) M N^{-\alpha m}\, ,
\quad
\text{for every $m>0$}\, .
\]
.
\end{theorem}

\begin{remark}\label{r:marginali?-2}
A remark here is in order. 
Note that in Theorem \ref{t:loop-equation-discretized-2} we are investigating what happens to a single trajectory $t\mapsto (P_0 (t), \ldots , P_{N-1} (t))$ when we apply the discretized operator which approximates the infinite-dimensional operator appearing in \eqref{e:weak-loop-equation}.    
However, the discretized momentum variables $P_k$ arose when taking a Fourier transform, see Section \ref{subsec:discretizedmomentum} , giving us a ``time-evolving generalized function $\beta_t$'' in $P_k$ as in \eqref{e:vero-beta}. Assuming for the moment that $\beta_t$ has enough regularity to be at least a ``signed measure'', then the computation in Theorem \ref{t:loop-equation-discretized-2} is relevant if we know that, at least up to some further error terms in $N^{-1}$, $\beta_t$ is the time marginal of a measure $\beta$ which is concentrated on a particular family of trajectories $\{P_k (t)\}$, which has enough regularity to be differentiated at least once in the variable $t$. Whether and when we can expect such a property to hold is a very challenging problem, which is the analog of the one raised in Remark \ref{r:marginali?}.  
\end{remark}

\begin{proof}
Fix $0\le k \le N-1$. It suffices to show that
\begin{equation}
    (\Omega_{k+3} \times U_k   - \nu \mathcal{D}_k)\Psi = (E_k + R_{\rm mo}^k) \Psi
\end{equation}
where $|R_{\rm mo}^k| \le \const(m,\gamma,\nu,\Lambda) N^{-\alpha m}$ for every $m\ge 1$.

Using the identity
\begin{equation}
\Psi [P, C, t]
= \exp \left\{ - \frac{i\gamma}{\nu} \sum_{k=0}^{N-1} C_k \cdot \triangle P_{k-1}(t)\right\}\, ,
\end{equation}
we easily obtain the following equalities
\begin{equation}
\begin{split}
        \Omega_k \Psi &= - \frac{i \gamma}{\nu} \triangle P_k \times \triangle P_{k-1} \, \Psi
        \\
        \mathcal{D}_k \Psi &= -\frac{2 \gamma^2}{\nu^2} \triangle P_{k-1} \times (\triangle P_k \times \triangle P_{k-1})\, \Psi
        \\& = - \frac{2\gamma^2}{\nu^2} (\triangle P_k\, |\triangle P_{k-1}|^2 - \triangle P_{k-1}\, (\triangle P_{k-1}\cdot \triangle P_k))\, \Psi
\end{split}
\end{equation}
By  Proposition \ref{prop:simple_waves}, we obtain
\begin{equation}
\begin{split}
   U_k \Psi &= - \frac{1}{\log(N)} \frac{\triangle P_{k-1}}{|\triangle P_{k-1}|^2} \times (\triangle P_k \times \triangle P_{k-1}) (1 + R(N^\alpha \gamma \nu^{-1} \triangle P_{k-1}))\,\Psi  
   \\&
   =  - \frac{1}{\log(N)} \left( \triangle P_k - \frac{\triangle P_{k-1}\, (\triangle P_{k-1}\cdot \triangle P_k)}{|\triangle P_{k-1}|^2}\right)(1 + R(N^\alpha \gamma \nu^{-1} \triangle P_{k-1}))\,\Psi
\end{split} 
\end{equation}
where $R\in \mathcal{S}$, is a fixed Schwartz function.
We have all the ingredients to make the following explicit:
\begin{equation*}
\Omega_{k+3} \times U_k \Psi 
 =
\frac{i\gamma}{\nu\log(N)} (\triangle P_{k+3}  \times \triangle P_{k+2}) \times \left( \triangle P_k - \tfrac{\triangle P_{k-1}\, (\triangle P_{k-1}\cdot \triangle P_k)}{|\triangle P_{k-1}|^2}\right)\Psi 
+ R_{\rm mo}^k \Psi
\end{equation*}
where, $R^k_{\rm mo}$ satisfies the following estimate
\begin{equation}
\begin{split}
    |R^k_{\rm mo}| &\le \frac{2}{\log(N)}\frac{\gamma}{\nu}|\triangle P_{k+3}||\triangle P_{k+1}||\triangle P_k| |R(N^\alpha\gamma\nu^{-1}\triangle P_{k-1})| 
    \\& \le \const(m) \frac{\gamma}{\nu}\frac{\Lambda^3}{\log(N)}\frac{1}{(1 + N^\alpha \gamma \nu^{-1}\Lambda^{-1})^m}
    \\& \le \const(m,\gamma,\nu,\Lambda) N^{-\alpha m},
\end{split}
\end{equation}
for every $m\ge 1$, as $R\in \mathcal{S}$.
\end{proof}

\subsection{Comparison to Migdal's Discretized Momentum Loop Equation}\label{s:discussione}
Following Migdal, the next step is to take the main expression in \eqref{e:altra-formula} and set it to zero for all $C$'s, in order to derive an equation for a of evolving curves $P_k (t)$, $k\in \{0, \ldots , N-1\}$. This leads to the system
\begin{equation}\label{e:sistema}
\frac{d}{dt} P_k (t) - E_k (t) = {\rm const}\, ,
\end{equation}
where ${\rm const}$ is a ``constant vector'', in other words a vector independent of $k$.

In Migdal's theory, however, the corresponding analogous equations take a different form. In order to discuss the main difference we will set the constant $\nu=1$ and retain rather the dependence upon $\gamma$, which clearly highlights two different terms in \eqref{eq:E_k}, one which is linear in $\gamma$ and one which is quadratic. We report it here the formula for the reader's convenience:
\begin{equation}\label{e:nostra}
\begin{split}
    E_k := \frac{1}{\log(N)}& i \gamma (\triangle P_{k+3} \times \triangle P_{k+2}) \times \left( \triangle P_k - \frac{\triangle P_{k-1}\, (\triangle P_{k-1}\cdot \triangle P_k)}{|\triangle P_{k-1}|^2}\right)
        \\& + \gamma^2 (\triangle P_k\, |\triangle P_{k-1}|^2 - \triangle P_{k-1}\, (\triangle P_{k-1}\cdot \triangle P_k))\, .
\end{split}    
\end{equation}
One technical discrepancy which does not affect the final outcome is that in the derivation of the discretized equations we use a ``regularized Biot-Savart'' operator at the the level of $N$, while Migdal uses (formally) the usual Biot-Savart operator. This means that while we get the additional error term $R_{\rm mo}$ in \eqref{e:altra-formula}, in the corresponding computation Migdal gets instead an exact identity.

\medskip

The really relevant difference stems from the use of the two different discretizations explained in the previous sections. Given his choice, Migdal in \cite{Migdal24b} derives the following system 
\begin{equation}\label{e:Sasha-system}
\frac{d}{dt} \frac{P_k (t) + P_{k-1} (t)}{2} - \mathcal{E}_k (t) = 0\, ,
\end{equation}
where the terms $\mathcal{E}$ take the form
\begin{equation}\label{e:Sasha}
\begin{split}
\mathcal{E}_k := &  \frac{i \gamma}{\log N}\Bigg(\frac{\big({\textstyle{\frac{P_k+P_{k-1}}{2}}}\cdot \triangle P_{k-1}\big)^2}{|\triangle P_{k-1}|^2} - \Big(\frac{P_k + P_{k-1}}{2}\Big)\cdot  \Big(\frac{P_k + P_{k-1}}{2}\Big)\Bigg) \triangle P_{k-1}\\
& \quad + \gamma^2 \Big(\big({\textstyle{\frac{P_k+P_{k-1}}{2}}}\cdot \triangle P_{k-1}\big) \triangle P_{k-1} - |\triangle P_{k-1}|^2 {\textstyle{\frac{P_k+P_{k-1}}{2}}}\Big)\, 
\end{split}
\end{equation}
We note that the logarithmic correction in the term which is linear in $\gamma$ is absent in Migdal's earlier works (see e.g. \cite{Migdal23}) and it appears for the first time in \cite{Migdal24b}. This correction was in fact introduced by Migdal after preliminary discussions with the authors of this paper about the effect of the Biot-Savart operator in the far field, which was initially ignored in \cite{Migdal23}.

\begin{remark}\label{r:non-hermite}
We warn the reader that, while we are imposing that $\triangle P_k$ is real, we are not requiring the same property upon $P_k$. Moreover the ``scalar products'' $(P_k+P_{k-1}) \cdot (P_k+P_{k-1})$ and $(P_k+P_{k-1}) \cdot \triangle P_k$ denote the sums of the products of the (complex) components of the corresponding vectors and {\em not} the usual hermitian product between the two complex vectors. In particular
\[
(P_k+P_{k-1}) \cdot (P_k+P_{k-1}) \neq |P_k+P_{k-1}|^2\, .
\]
On the other hand, given \eqref{e:tecnica}, we do have $\triangle P_k \cdot \triangle P_k = |\triangle P_k|^2$. 
\end{remark}

Migdal then finds a particular family of solutions of \eqref{e:Sasha-system} which decouples the time variable and the discrete variable $k$. More precisely, let us postulate the existence of 
a solution to the system \eqref{e:Sasha-system} of the form
\[
P_k (t) = \sqrt{\frac{1}{2(t+t_0)}} \frac{G_k}{\gamma}\, ,
\]
where the $G_k$ are ``constant'' vectors, i.e. they do not depend on time. We then following Migdal again and introduce the new variables
\[
F_k := \frac{G_k + G_{k-1}}{2}
\]
and 
\[
\triangle F_k = G_k - G_{k-1}\, ,
\]
and we also impose, consistently with \eqref{e:tecnica}, that 
\begin{equation}\label{e:reale}
{\rm Im}\, \triangle F_k = 0 \quad\mbox{and}\quad
{\rm Re}\, \triangle F_k \neq 0\, .
\end{equation}
We then find the system of algebraic equations
\begin{equation}\label{e:sistema-giusto}
\left[\frac{1}{ \log(N)} \left( \frac{(F_k \cdot \triangle F_k)^2}{|\triangle F_k|^2} - F_k \cdot F_k\right)
    - i \gamma(F_k \cdot \triangle F_k) \right]
    \triangle F_k
    - i \gamma (1 - |\triangle F_k|^2) F_k  = 0.
\end{equation}
\begin{remark}\label{r:non-hermite-2}
As in Remark \ref{r:non-hermite}, we stress that we are imposing $\triangle F_k\in \mathbb R^3$, but that $F_k$ might have nonzero imaginary part: this is indeed a key point in Migdal's discretized Euler ensemble. Therefore we do not have
\[
F_k \cdot F_k = |F_k|^2\, ,
\]
and we will indeed see below that for the actual Euler ensemble $F_k \cdot F_k$ vanishes even though $F_k \neq 0$!
On the other hand, given \eqref{e:reale}, we still have $\triangle F_k \cdot \triangle F_k = |\triangle F_k|^2$. 
\end{remark}
If we introduce the notation 
\begin{align*}
I_a &:= \Big(\frac{(F_k \cdot \triangle F_k)^2}{|\triangle F_k|^2} - F_k \cdot F_k\Big) \triangle F_k\\
I_b &:= (F_k \cdot \triangle F_k) \triangle F_k\\
I_c &:= (1-|\triangle F_k|^2) F_k\, .
\end{align*}
we can then rewrite \eqref{e:sistema-giusto} as 
\begin{equation}\label{e:giusto-2}
(\log (N))^{-1} I_a - i\gamma I_b - i \gamma I_c = 0\, .
\end{equation}
This final equation must be compared to \cite[Eq. (47)]{Migdal23}, which, using the notation above, takes instead the form
\begin{equation}\label{e:Sasha-system-2}
 I_a - i\gamma I_b - i \gamma I_c = 0\, .
\end{equation}
However, following \cite[Section 3]{Migdal23}, the special set of solutions found by Migdal in  \cite[Section 3]{Migdal23}, which amounts to a particular choice of the vectors $F_k$, solve in fact the three equations
\begin{equation}\label{e:tutti-zero}
I_a = I_b = I_c = 0\, .
\end{equation}
More precisely, Migdal sets $G_k$ to be equal to $iA + f_k$ for suitable {\em real} vectors $A, f_k \in \mathbb R^3$ and imposes the following set of conditions on them (cf. \cite[(125), (126), (127), and (129a)]{Migdal23}):
\begin{align}
|f_k-f_{k-1}|^2 &= 1\label{e:modulo-1}\\
A\cdot f_k &=0 \label{e:ortog}\\
|f_k|^2 &= |f_{k-1}|^2 \label{e:ortog2}\\
4|A|^2 &= |f_k+f_{k+1}|^2\, .\label{e:nascosta}
\end{align} 
From these conditions it follows easily that
\begin{itemize}
\item $I_c =0$ because
\[
|\triangle F_k|^2 = |f_k -f_{k-1}|^2 \stackrel{\eqref{e:modulo-1}}{=} 1\, ;
\]
\item $I_b =0$ because 
\[
F_k \cdot \triangle F_k = i A \cdot f_k - i A \cdot f_{k-1} + |f_k|^2 - |f_{k-1}|^2 \stackrel{\eqref{e:ortog}\&\eqref{e:ortog2}}{=} 0\, ;
\]
\item $I_c =0$ because $F_k \cdot \triangle F_k =0$ and 
\[
F_k \cdot F_k \stackrel{\eqref{e:ortog}}{=} - |A|^2 + \frac{|f_{k+1} + f_k|^2}{4} \stackrel{\eqref{e:nascosta}}{=} 0\, .
\] 
\end{itemize}
Finally, cf. \cite[Theorem 1, Section 13.2]{Migdal23}, solutions to the equations \eqref{e:modulo-1}-\eqref{e:nascosta} are found explicitly. They have, however, a quite elegant geometric interpretation: the vectors $f_k$ turn out to be the vertices of a star polygon with the Schl\"afli symbol $\{q/p\}$, inscribed in a planar circle that is orthogonal to the vector $A$. These computations can be summarized in the following.

\begin{theorem}\label{t:Sasha-finale}
Let $\{P_k\}$ be an element in the discretized Euler ensemble described above and consider the functional $\Psi$ in \eqref{e:Psi-in-P}. Then the discretized Euler ensemble satisfies \eqref{e:Sasha-system}. 
\end{theorem}

We finish this discussion with a mention of the latest work \cite{Migdalnuovo}. In the latter Migdal claims that, using the liquid loop approach, the corresponding discretized equation in the momentum variables reduces to the nullification of the operator 
\begin{equation}\label{e:operatore-liquido}
\Big(\partial_t + i\gamma \sum_{k=1}^{N-2} \mathcal{D}^M_k\Big)\, ,
\end{equation}
up to error terms which vanish in the limit $N\to \infty$. The corresponding equation in the discretized momentum variables for the latter operator is 
\begin{equation}\label{e:Sasha-liquido}
 \frac{d}{dt} \frac{P_k (t) + P_{k-1} (t)}{2} - \mathcal{E}^\ell_k (t) = 0\, ,
\end{equation}
where the terms $\mathcal{E}^\ell_k$ take the form
\begin{equation}\label{e:Sasha-liquido-2}
\mathcal{E}^\ell_k = \gamma^2 \Big(\big({\textstyle{\frac{P_k+P_{k-1}}{2}}}\cdot \triangle P_{k-1}\big) \triangle P_{k-1} - |\triangle P_{k-1}|^2 {\textstyle{\frac{P_k+P_{k-1}}{2}}}\Big)
\end{equation}
In particular from the computations outlined above it follows readily that 

\begin{theorem}\label{t:Sasha-liquido}
The discretized Euler ensemble solves the system \eqref{e:Sasha-liquido}. 
\end{theorem}

On the other hand Lemma \ref{l:diffusion migdal} and the same arguments as in the proof of Theorem \ref{t:loop-equation-discretized} yield immediately the following theorem.

\begin{theorem}\label{t:loop-equation-discretized-3}
Fix an integer $N \ge 2$ and a parameter $M\geq 1$. 
Assume that $u \in C^1(\R^3 \times [0,T]) \cap C\big([0,T], C^3)$
is a divergence-free time-dependent vector field,
let $C \in \mathcal{L}_N \R^3$ be a polygonal loop satisfying \eqref{eq:cond M/N}, and define
\begin{equation}
\Psi[u,C,t] 
= \exp\!\left\{\frac{i\gamma}{\nu} \int_0^1 u(C(\theta),t) \cdot C'(\theta)\, d\theta\right\}.
\end{equation}
Then
\begin{equation}\label{eq:Psi-versione-3}
\left(\partial_t  
+ i \gamma \sum_{k=1}^{N-2} \triangle C_k \cdot \mathcal{D}^M_k\right) 
\Psi
= \Psi \int_0^1 \big[\partial_t v + \nu \nabla \times \omega\big] (C (\theta), t) \cdot C' (\theta)\, d\theta + R_{\mathrm{loop}},
\end{equation}
where the remainder satisfies
\begin{equation}\label{e:loop-error-3}
|R_{\mathrm{loop}}| 
\le \const(M,\nu,\gamma, \|u\|_{C^3})\, N^{-1}\, .
\end{equation}
\end{theorem}

\appendix

\section{Regularized Biot-Savart law}
\label{sec:BSreg}

The following proposition demonstrates that the operator $\mathrm{BS}_\ell$ closely approximates the standard Biot-Savart operator when the input velocity field is sufficiently integrable and $\ell$ is sufficiently small.

\begin{proposition}\label{prop:BS_Ncurl}
Fix $K\ge 0$.
Let $\omega:= \nabla \times u$, where $u\in L^2_\sigma \cap H^{K+1}$. Then
\begin{equation}
\|{\rm BS}_\ell[\omega] - u\|_{H^K}
	\le 
C(K) \ell^{3/2} \| u \|_{H^K}\, .
\end{equation}
\end{proposition}

\begin{proof}
By \eqref{eq:BSN1}, it is sufficient to estimate the $H^K$ norm of \begin{equation}
		I(x):= \frac{1}{4\pi} \int_{\mathbb{R}^3} (\nabla \times u(y)) \times \nabla_y \left( \frac{1 - \chi(\ell (x-y))}{|x-y|}\right)\, dy \, .
	\end{equation}
For every $0\le k \le K$, we integrate by parts and use Holder inequality:
	\begin{align}
		|D^k I(x)| 
		& \le C \int_{\mathbb{R}^3} |D^k u(y)| \Big| \nabla^2 \left( \frac{1 - \chi(\ell(x-y))}{|x-y|}\right) \Big|\, dy
		\\&
		\le C \ell^3 \int_{B_{2\ell^{-1}}\setminus B_{\ell^{-1}}(x)} |D^k u(y)|\, dy 
		\\ & \quad \quad + C \int_{\mathbb{R}^3 \setminus B_{\ell^{-1}}(x)} \frac{|D^k u(y)|}{|x-y|^3}\, dy
		\\& \le C \ell^{3/2} \| D^k u \|_{L^2}.
	\end{align}
\end{proof}

The next proposition considers the action of ${\rm BS}_\ell$ on polynomial growth velocity fields.

\begin{proposition}\label{prop:polygrowth}
	Let us consider a velocity field satisfying
	\begin{equation}
		|\omega(x)| \le M (1 + |x|^m) \, , \quad x\in \mathbb{R}^3\, ,
	\end{equation}
	for some $M\ge 1$ and $m\ge 0$.
	Then, 
	\begin{equation}
		|{\rm BS}_\ell [\omega](x)|
		\le C(m) \ell^{-m-1} M(1+|x|^m)\, ,
		\quad x\in \mathbb{R}^3\, .
	\end{equation}
\end{proposition}

\begin{proof}
The proof is a straightforward estimate from \eqref{eq:BSN1}. For every $x\in \mathbb{R}^3$, we have
	\begin{align}
		|{\rm BS}_\ell[\omega](x)|
        & \le C \int_{\R^3} |\omega(x-y)| \Big|D \frac{\chi(\ell y)}{|y|}\Big|\, dy
		\\&\le  C M \int_{\mathbb{R}^3}(1 + |x-y|^m) \left( \frac{\chi(\ell y)}{|y|^{2}} +  \frac{\ell |D^k\chi|(\ell y)}{|y|}
		\right)\, dy
		\\& \le CM \int_{B_{\ell^{-1}}(0)} \frac{1 + |x-y|^m}{|y|^2}\, dy
		\\& \le C(m) M (\ell^{-m} + |x|^m) \int_{B_{\ell^{-1}}(0)} \frac{1}{|y|^2}\, dy
		\\& \le C(m) \ell^{-m-1} M(1+|x|^m).
	\end{align}
\end{proof}

We next compute the regularized Biot-Savart operator on complex exponentials 

\begin{proposition}\label{prop:simple_waves}
	Let $\omega(y) := v_0 e^{i a \cdot y}$ for $v_0, a \in \mathbb{R}^3 \setminus \{0\}$.  
	There exists a Schwartz function $R \in \mathcal{S}(\mathbb{R}^3)$ such that
		\begin{equation}\label{e:estimate-(1)}
			\mathrm{BS}_\ell [\omega](x) = i e^{i a \cdot x} \frac{a}{|a|^2} \times v_0 \left(1 + R(a \ell^{-1})\right).
		\end{equation}
\end{proposition}

Note that \eqref{e:estimate-(1)} is effective when $\ell^{-1} \gg |a|$, implying that $\mathrm{BS}_\ell [\omega](x) \approx i e^{i a \cdot x} \frac{a}{|a|^2} \times v_0$ up to an error of magnitude $C(k)(1 + |a|\ell^{-1})^{-k}$ for every $k\ge 1$.

\begin{proof}[Proof of Proposition \ref{prop:simple_waves}]
	From \eqref{eq:BSN2}, we deduce
	\begin{align*}
		{\rm BS}_\ell [\omega](x) =& \frac{1}{4\pi}\int_{\mathbb{R}^3} \frac{\chi(\ell (x-y))}{|x-y|}\, i a\times v_0\,  e^{i a y}\, dy
		\\& = i e^{i a \cdot x} \frac{a}{|a|^2} \times v_0\, \cdot \frac{1}{4\pi} \int_{\mathbb{R}^3}  \frac{\chi(\ell z)}{|z|}  |a|^2 e^{-i az}\, dz.
	\end{align*}
	We change the variables according to $z'=z \ell$:
	\begin{align*}
		\frac{1}{4\pi} \int_{\mathbb{R}^3}  \frac{\chi(\ell z)}{|z|}&  |a|^2 e^{-i az}\, dz
		\\& = 
		\frac{1}{4\pi} \int_{\mathbb{R}^3}  \frac{\chi(z')}{|z'|}  |a|^2N^{2\alpha} e^{-i \ell^{-1} az'}\, dz'
		\\& = 
		-\frac{1}{4\pi} \int_{\mathbb{R}^3}  \frac{\chi(z')}{|z'|}  \Delta_{z'}( e^{-i \ell^{-1} az'})\, dz'.
	\end{align*}
	Finally, we integrate by parts and use that $G(z) = \frac{1}{4\pi|z|}$ is the fundamental solution of the Laplacian:
	\begin{equation}\label{eq:est1}
		\begin{split}
			\frac{1}{4\pi} \int_{\mathbb{R}^3} & \frac{\chi(\ell z)}{|z|}  |a|^2 e^{-i az}\, dz 
			\\& = 1  
			- \frac{1}{4\pi} \int_{\mathbb{R}^3} \left( \frac{\Delta \chi(z')}{|z'|} - \frac{\nabla \chi(z')\cdot z'}{|z'|^3} \right) e^{-i \ell^{-1} a z'}\, dz'
		\end{split}
	\end{equation}
	 \eqref{e:estimate-(1)} follows by observing that the second term in \eqref{eq:est1} is the Fourier transform of a compactly supported smooth function.
\end{proof}


\begin{thebibliography}{GMS13} 
	
	
\bibitem[AF03]{adams2003sobolev}
\textsc{R. A. Adams, J. J. F. Fournier:}
\textit{Sobolev Spaces}, 2nd ed., 
Pure and Applied Mathematics (Amsterdam), Vol.~140, 
Elsevier/Academic Press, Amsterdam, 2003.

\bibitem[BKT22]{BKT}
{\sc Z. Bradshaw, I. Kukavica, and T.-P. Tsai:} 
{\it Existence of global weak solutions to the Navier-Stokes equations in weighted spaces.} 
Indiana Univ. Math. J., 71(1):191--212, 2022.

	\bibitem[CKN82]{caffarelli1982}
		{\sc L. Caffarelli, R. Kohn, L. Nirenberg:}
		{\it Partial regularity of suitable weak solutions of the Navier-Stokes equations.}
		Comm. Pure Appl. Math., 35(6):771--831, 1982.



\bibitem[F72]{Foias72}
		{\sc C. Foias:}
		{\it Statistical study of Navier–Stokes equations. I.}
Rend. Sem. Mat. Univ. Padova 48, 219–348 (1972)

\bibitem[F73]{Foias73}
		{\sc C. Foias:}
		{\it Statistical study of Navier–Stokes equations. II.}
Rend. Sem. Mat. Univ. Padova 49, 9–123 (1973)

\bibitem[Ga25]{Galeati}
{\sc L. Galeati:}
{\it Almost-everywhere uniqueness of Lagrangian trajectories for 3D Navier–Stokes revisited.}
J. Math. Pures Appl., 200 (2025).


\bibitem[HOUZ10]{HOUZ10}
		{\sc H. Holden, B. Oksendal, J. Ubo, T. Zhang:}
		{\it Stochastic partial differential equations}
Springer, New York (2010).



\bibitem[H52]{Hopf52}
		{\sc E. Hopf:}
		{\it Statistical hydromechanics and functional calculus}
J. Ration. Mech. Anal. 1, 87–123 (1952).



\bibitem[LV76]{LadyVers76}
		{\sc A.-M. Ladyzenskaja, O. A. Versik:}
		{\it  The evolution of measures that are defined by Navier-Stokes equations, and the solvability of the Cauchy problem for the statistical equation of E. Hopf.}
Zap. Naucn. Sem. Leningrad. Otdel. Mat. Inst. Steklov. (LOMI) 59 (1976), 3–24, 255.


	\bibitem[L34]{Leray34}
		\textsc{J. Leray:}
		\textit{Sur le mouvement d’un liquide visqueux emplissant l’espace}.
		Acta Mathematica, {\bf 63} (1934), 193--248.

\bibitem[MB]{MajdaBertozzi}
\textsc{A. J. Majda, A. L. Bertozzi:}
\textit{Vorticity and incompressible flow.}
Cambridge Texts in Applied Mathematics. Cambridge University Press, 2002.

\bibitem[M23]{Migdal23}
		{\sc A. Migdal:}
		{\it To the theory of decaying turbulence.}
Fractal and Fractional 7 (10), 754.

\bibitem[M24]{Migdal24}
		{\sc A. Migdal:}
		{\it Quantum Solution of classical turbulence. Decaying energy spectrum.}
Physics of Fluids 36 (9), 095161.


\bibitem[M24b]{Migdal24b}
		{\sc A. Migdal:}
		{\it Duality of Navier-Stokes to a one-dimensional System.}
ArXiv:2411.01389  

\bibitem[M25]{Migdal25}
{\sc A. Migdal:}
{\it Private communications to E. Bru\'e and C. De Lellis.}

\bibitem[M25b]{Migdalnuovo}
{\sc A. Migdal:} 
{\it Microscopic theory of turbulent mixing: discrete shell structures in scalar concentration.
ArXiv:2504.10205v5
}


\bibitem[NPS13]{NPS13}
		{\sc A. Nahmod, N. Pavlovic, G. Staffilani:}
		{\it Almost sure existence of global weak solutions for supercritical Navier-Stokes equations.}
SIAM J. Math. Anal. 45 (2013), no. 6, 3431–3452.

\bibitem[RS09]{RS09}
{\sc J. C. Robinson, W. Sadowski:}
{\it Almost-everywhere uniqueness of Lagrangian trajectories for suitable weak solutions of the three-dimensional Navier-Stokes equations.}
Nonlinearity 22 (2009), no. 9, 2093--2099.

\bibitem[RS09b]{RS09b}
{\sc J. C. Robinson, W. Sadowski:}
{\it A criterion for uniqueness of Lagrangian trajectories for weak solutions of the 3D Navier-Stokes equations.}
Comm. Math. Phys. 290 (2009), no.1, 15--22.

\bibitem[S76]{Scheffer76}
		{\sc V. Scheffer:}
		{\it Partial regularity of solutions to the Navier-Stokes equations.}
Pacific J. Math. 66 (1976), no. 2, 535–552.		

\bibitem[S77]{Scheffer77}
		{\sc V. Scheffer:}
		{\it Hausdor measure and the Navier-Stokes equations.}
Comm. Math. Phys. 55 (1977), no. 2, 97–112.



\bibitem[VF78]{VishikFursikov78}
		{\sc M.-I. Vishik, A.-V. Fursikov:}
		{\it Translationally homogeneous statistical solutions and individual solutions with infinite
energy of a system of Navier–Stokes equations.}
Sibirsk. Math. Zh. 19(5), 1005–1031 (1978)


\bibitem[VF88]{VishikFursikov88}
		{\sc M.-I. Vishik, A.-V. Fursikov:}
		{\it Mathematical problems of statistical hydromechanics.}
Mathematics and Its Applications
(Soviet Series), 9. Kluwer Academic Publishers, Dordrecht (1988)

\bibitem[Li13]{Lions-book}
  {\sc P.-L. Lions:}
  {\it Mathematical topics in dluid Mechanics: Volume 1: Incompressible models.}
  Oxford University Press, Oxford (2013)

\bibitem[Zi89]{Ziemer}
{\sc W.-P. Ziemer:}
{\it Weakly differentiable functions: Sobolev spaces and functions of bounded variations.}
Graduate Texts in Mathematics, 120. Springer Verlag, Berlin (1989).


\end{thebibliography}
\end{document}